\documentclass[10pt, reqno]{amsart}
\usepackage{amsfonts,amssymb,amscd,amsmath,enumerate,verbatim,calc}
\usepackage{mathrsfs}
\usepackage{tikz-cd}
\numberwithin{equation}{section}
\newtheorem{theorem}{Theorem}[section]
\newtheorem{corollary}[theorem]{Corollary}
\newtheorem{lemma}[theorem]{Lemma}

\newtheorem{proposition}[theorem]{Proposition}
\newtheorem{definition}[theorem]{Definition}

\newtheorem{example}[theorem]{Example}
\newtheorem{remark}[theorem]{Remark}

\begin{document}
% %----------------------------------------------------------------------
%%% ----------------------------------------------------------------------
%%% ----------------------------------------------------------------------

\title[Quadratic and Symplectic structures on 3-(Hom)-$\rho$-Lie algebras]
 {Quadratic and Symplectic structures on 3-(Hom)-$\rho$-Lie algebras}

%%% ----------------------------------------------------------------------
%%% ----------------------------------------------------------------------
%%% ----------------------------------------------------------------------
\bibliographystyle{amsplain}
%%% ----------------------------------------------------------------------
%%% ----------------------------------------------------------------------
%%% ----------------------------------------------------------------------

\author[]{Zahra Bagheri ${}^{1}$ and Esmaeil Peyghan ${}^{2, *}$\\
${}^{1, 2}$ Department of Mathematics, Faculty of Science, Arak University,\\
	Arak, 38156-8-8349, Iran\\
${}^{*} \lowercase{Correspondence: e-peyghan@araku.ac.ir}$
}



%%% ----------------------------------------------------------------------
%%% ----------------------------------------------------------------------
%%% ----------------------------------------------------------------------

\begin{abstract}
Our purpose in this paper is the generalization of the notions of quadratic and symplectic structures to the case of 3-(Hom)-$\rho$-Lie algebras. We describe some properties of them by expressing the related lemmas and theorems. Also, we introduce the concept of 3-pre-(Hom)-$\rho$-Lie algebras and define their representation.
\end{abstract}

%%% ----------------------------------------------------------------------
%%% ----------------------------------------------------------------------
%%% ----------------------------------------------------------------------
\maketitle
%%% ----------------------------------------------------------------------
%%% ----------------------------------------------------------------------
%%% ----------------------------------------------------------------------

%%%%%%%%%%%%%%%%%%%%%%%%%%%%%%%%%%%%%%%%%%%%%%%%%%%%%%%%%%%%%%%%%%%%%%%%%%%%%%%%%%%%%%

%%%%%%%%%%%%%%%%%%%%%%%%%%%%%%%%%%%%%%%%%%%%%%%%%%%%%%%%%%%%%%%%%%%%%%%%%%%%%%%%%%%%%%

%%% ----------------------------------------------------------------------
%%% ----------------------------------------------------------------------
%%% ----------------------------------------------------------------------
%%% ----------------------------------------------------------------------
%%% ----------------------------------------------------------------------

%\newpage

%%%%%%%%%%%%%%%%%%%%%%%%%%%%%%%%%%%%%%%%%%%%%%%%%%%%%%%%%%%%

%%%%%%%%%%%%%%%%%%%%%%%%%%%%%%%%%%%%%5
%%%%%%%%%%%%%%%%%%%%%%%%%%%%%%%%%%%%%

\section{Introduction} 
At the first, Hartwig, Larsson and Silvestrov in \cite{J D S} generalized Lie algebras and constructed a new class of Lie algebras in the name of Hom-Lie algebras. These algebras very soon were attracted considerable attentions and many algebraic and geometric structures were introduced on them.  In 1994, the other concept was emerged as "$\rho$-Lie algebra" or "Lie color algebra" by P.J. Bongaarts \cite{BP1} and then in 1998, Scheunert and Zhang introduced a new notion which is called a cohomology theory of Lie color algebras \cite{SZ}. Also, the notion of Hom-Lie color algebras, as an extension of Hom-Lie superalgebras to $G$-graded algebras, was introduced by Yuan \cite{LY}. These Lie algebras were applied by the researchers and in 2015, Abdaoui, Ammar and Makhlouf defined representations and a cohomology of Hom-Lie color algebras in \cite{AAM}. After two years, in 2017, B. Sun et al. studied some concepts such as $T^*$-extension and abelian extension on Hom-Lie color algebras \cite{BLY}.

The notion of pre-Lie algebra was given by Gerstenhaber in \cite{MG} in the
study of deformations and the cohomology theory of associative algebras. These algebras in the other research articles appear with titles such as left-symmetric algebras or right-symmetric algebras. Also, the notion of  Hom-pre-Lie algebra was introduced in \cite{A S}. This kind of algebras in the construction of Hom-Lie 2-algebras has an important role (see \cite{Y C}). Representation of a Hom-pre-Lie algebra was constructed in \cite{SL}, in the study of Hom-pre-Lie bialgebras. 
The geometrization and the bialgebra theory of these algebras was
studied in \cite{SL, ZY}. 
For the commutator bracket $[.,.]_c$, every Hom-pre-Lie algebra $(A,\cdot,\alpha)$ gives rise to a Hom-Lie algebra $(A, [.,.]_c, \alpha)$. 
This Lie algebra is called the subadjacent Hom-Lie algebra and denoted by $A^c$.
Actually, we can say the Hom-pre-Lie algebras have a interconnected with the Hom-Lie algebras.

The concept of $n$-Lie algebra was introduced by Filippov in 1985. These Lie algebras are known by different names such as Filippov algebras, Nambu-Lie algebras, Lie $n$-algebras. The special case of these algebras are 3-Lie algebras and in the Hom case 3-Hom-Lie algebras, which extracted from 3-Lie algebras base on Hom-Lie algebras. These algebras were defined and checked by H. Ataguema et al. in \cite{HAS}. Y. Liu et al. took into consideration these algebras and  studied the representations and module-extensions on them in \cite{LCM}. Later, the new notions such as 3-Lie color algebra, 3-Hom-Lie color algebra and 3-pre-Lie algebra were constructed and studied in \cite{ CLY,PZCG,T}. In 1996, Y. Daletskii and V. Kushnirevich introduced the notion of $n$-Lie superalgebra, as a natural generalization of $n$-Lie algebra (\cite{DK}). When $n=3$, the 3-Lie superalgebras are a particular case of 3-$\rho$-Lie algebras when the abelian group is $\mathbb{Z}_2$ (see for more  details \cite{VA, VA1, GCS}).

Ternary Lie algebras have vast applications in various realms of modern physics and mechanics such as Nambu, field and  Hamiltonian theories. For example,  in Nambu consideration for a case that the algebra of functions in a manifold $M$  is ternary, one can construct a bracket defined by  a trivector on $M$ that is famous as  Nambu bracket in literature. In this way, the dynamics of $M$ is controlled by two Hamiltonians in return of classical case. Indeed, $3$-ary bracket in ternary Lie algebra is a suitable alternative for Poisson alternative in Hamiltonian mechanics. To study how Ternary Lie algebras can model this frameworks, see \cite{Gautheron0, Nambu0, TL}. Nahm equations describing some facts in string and superstring theories can be formulized using ternary Lie algebras also. In \cite{BH} it is described the benefits of replacing ternary Lie algebras instead of using simple Lie algebras in lifted Nahm equations and in \cite{JN} the authors for the lifted Nahm equations replaced the Lie algebra in the Nahm equation by a 3-Lie algebra. In the framework of Bagger-Lambert Gustavsson model of multiple $M2$ brains, it is possible to revisit the super-symmetric alternative of world-volume theory of $M$-theory membrane in special cases such $N = 2$.  See \cite{JN, BL1, HH, PP0} for more details in the way.

In 2006, quadratic Lie algebras were discussed by I. Bajo et al. in \cite{BBM}. A quadratic Lie algebra comes equipped with a non-degenerate, symmetric and invariant bilinear form. Quadratic Lie algebras are important  in Mathematics and Physics specially in conformal field theory \cite{MS}. S. Benayadi and A. Makhlouf, in 2010 studied the quadratic Hom-Lie algebras \cite{BM} as a generalization of the classical case. This process continued until a generalization to the case of quadratic
(even quadratic) color Lie algebras was obtained in \cite{WZ}. Also, quadratic color Hom-Lie algebras were introduced and studied by F. Ammar et al. in \cite{FIS}. Also, the quadratic-algebraic structures have been studied on the class of $\Gamma$-graded Lie algebras specially in the case where $\Gamma=\mathbb{Z}_2$. These algebras are called homogeneous (even or odd) quadratic Lie superalgebras (see \cite{ABB, ABBB, HS, SZ}).

In this paper, we discuss the quadratic and symplectic structures on 3-(Hom)-$\rho$-Lie algebras and study the representations of them. The other aim is to introduce the notion of 3-pre-(Hom)-$\rho$-Lie algebras and check their phase spaces. Also, we study the relationship between the 3-pre-(Hom)-$\rho$-Lie algebras and 3-(Hom)-$\rho$-Lie algebras.

This paper is arranged as follows: In Section 2, we give a brief introduction to 3-$\rho$-Lie algebras and study the quadratic and symplectic structures on 3-$\rho$-Lie algebras. Also, we study the metrics and derivations on 3-$\rho$-Lie algebras and we explore the relationship between the symplectic structure and the invertible derivation which is antisymmetric with respect to a metric. Section 3 is devoted to study of representation theory on 3-$\rho$-Lie algebras and some classiflcation results will be given. In Section 4, we introduce the notion of 3-pre-$\rho$-Lie algebras and discuss some basic properties of 3-pre-$\rho$-Lie algebras including representation theory and some classiflcation results.
Also, we define $\rho$-$\mathcal{O}$-operator associated to the representation $(V,\mu)$ of a 3-$\rho$-Lie algebra  and construct a 3-$\rho$-pre-Lie algebra structure on the representation $V$ by a $\rho$-$\mathcal{O}$-operator and bilinear map $\mu$. This section also contains the notion of a phase space of a 3-$\rho$-Lie algebra and it will be shown that a 3-$\rho$-Lie algebra has a phase space if and only if it is sub-adjacent to a 3-$\rho$-pre-Lie algebra.
In Section 5, we study the constructions of Section 2, such as quadratic, symplectic structures, 3-associative Lie algebras in the case of 3-Hom-$\rho$-lie algebras. The representation theory of 3-Hom-$\rho$-lie algebras is also included in this section. Section 6 is devoted to the 3-pre-Hom-$\rho$-Lie algebras, which are studied similar to the classical case in Section 4. Actually, we give the similar results to Section 4 in the case of 3-pre-Hom-$\rho$-Lie algebras. 
%%%%%%%%%%%%%%%%%%%%%%%%%%%%%%%%%%%%%%%%%%%%%%%%%%%%%%%%%%%%%%%%%%%%%%%%%%%%%%%%%%%%%%

%%% -----------------------------------------------------------------------
%%% ----------------------------------------------------------------------
%%% ----------------------------------------------------------------------
%%% ----------------------------------------------------------------------
%%% ----------------------------------------------------------------------
%%% ----------------------------------------------------------------------

%\newpage
\section{quadratic 3-$\rho$-lie algebras}
We begin this section by describing the class of $\rho$-Lie algebras, which includes the special multiplication that is known as bracket and a two-cycle. This Lie algebras is called 3-$\rho$-Lie algebra. We summarize some definitions and introduce the notions such as quadratic structure, symplectic structure $\omega$, metric $\varphi$, derivation $D$ and $\varphi$-antisymmetric derivation $D$ on 3-$\rho$-Lie algebras and show that a symplectic structure $\omega$ may be defined on a 3-$\rho$-Lie algebra if and only if there exists a $\varphi$-antisymmetric invertible derivation $D$. Also, We give a brief exposition of 3-associative algebras and give the constructions of direct product 3-$\rho$-Lie algebra structures from 3-associative algebras.\\

Let us consider $(G, +)$ as an abelian group over a field $\mathbb{K}$ ($\mathbb{K} = \mathbb{R}$ or $\mathbb{K} = \mathbb{C}$). A map
$\rho:G\times G\longrightarrow \mathbb{K}^{\star}$ with the properties
\begin{align}
&\rho(a,b) =\rho(b,a)^{-1},\quad a,b\in G,\\
&\rho(a+b,c) =\rho(a,c)\rho(b,c),\quad a,b,c\in G,
\end{align}
is called a two-cycle. From two above properties we can deduce $\rho(a,b)\neq 0$, $\rho(0,b)=1$ and $\rho(c,c)=\pm 1$ for all $a,b,c\in G$ , $c\neq 0$. Note that if $\mathcal{B}$ is a $G$-graded vector space ($\mathcal{B}=\oplus_{a\in G}\mathcal{B}_a$), an element $f\in \mathcal{B}_a$ is called homogeneous of $G$-degree $a$. We denote by $Hg(\mathcal{B} )$ the set of homogeneous elements in $\mathcal{B}$ and by $|f|$ the $G$-degree of $f\in Hg(\mathcal{B} )$

 A $\rho$-Lie algebra is a triple $(\mathcal{B}, [.,.]_{\mathcal{B}},\rho)$, where 
 \begin{enumerate}
 \item[i.]
 $\mathcal{B} $ is a $G$-graded vector space,
 \item[ii.]
 $\rho$ is a two-cycle,
 \item[iii.]
$[.,.]_{\mathcal{B} }$ is a bilinear map on $\mathcal{B}$ such that
\begin{align*}
&\bullet |[f,g]_{\rho}|= | f|+| g|,\\
&\bullet [f, g]_{\rho} = -\rho(|f|,|g|)[g, f]_{\rho},\\
	&\bullet \rho(|h|,|f|)[f, [g, h]_{\rho}]_{\rho}+\rho(|g|,|h|)[h, [f, g]_{\rho}]_{\rho}+\rho(|f|,|g|)[g, [h, f]_{\rho}]_{\rho} = 0.
	\end{align*} 
\end{enumerate}
	The third condition is equivalent to 
	$$[f, [g, h]_{\rho}]_{\rho}=[[f, g]_{\rho},h]_{\rho}+\rho(|f|,|g|)[g, [f, h]_{\rho}]_{\rho}.$$
In the next, for simplicity we use $\rho(f,g)$ instead of $\rho(|f|,|g|)$.
\begin{definition}
A triple $ (\mathcal{B}, [.,.,.]_{\mathcal{B} }, \rho)$ is called a 3-$\rho$-Lie algebra if $\mathcal{B}$ is a $G$-graded vector space, 
$[.,.,.]_{\mathcal{B} }: \mathcal{B} \times \mathcal{B}\times \mathcal{B} \longrightarrow \mathcal{B}$ is a 3-linear map,
$\rho$ is a two-cycle and moreover the following conditions are satisfied
\begin{align*} 
&(1)~~|[f_1,f_2,f_3]_{\mathcal{B} }|=|f_1|+ |f_2| +|f_3|,\\
&(2)~~[f_1,f_2,[g_1,g_2,g_3]_{\mathcal{B} }]_{\mathcal{B} }=[[f_1,f_2,g_1]_{\mathcal{B} }, g_2, g_3]_{\mathcal{B} }+\rho(f_1+f_2,g_1)[g_1,[f_1,f_2,g_2]_{\mathcal{B} },g_3]_{\mathcal{B} }\\
&\ \ \ \ \ \ \ \ \ \ \ \ \ \ \ \ \ \ \ \ \ \ \ \ \ \ \ \ \ \ \ +\rho(f_1+f_2, g_1+g_2)[g_1,g_2,[f_1,f_2,g_3]_{\mathcal{B} }]_{\mathcal{B} },~~ \text{($\rho$-fundamental identity)}\\
& (3)~~ [f_1,f_2,f_3]_{\mathcal{B}}=-\rho(f_1,f_2)[f_2,f_1,f_3]_{\mathcal{B}}=-\rho(f_2,f_3)[f_1,f_3,f_2]_{\mathcal{B}}.~~ (\text{$\rho$-skew-symmetry property})
\end{align*}
We denote the 3-$\rho$-Lie algebra $ (\mathcal{B}, [.,.,.]_{\mathcal{B} }, \rho)$ by $\mathcal{B}_{\rho}$.\\
\end{definition}
\begin{example}
One of the most important case of 3-$\rho$-Lie algebras are 3-ary Lie superalgebras when abelian group $G$ is $\mathbb{Z}_2$.
\end{example}
In the following, we extend to 3-$\rho$-Lie algebras the study of quadratic color Hom-Lie algebras introduced in \cite{FIS} and in the next of this paper we investigate Hom version. We apply a non-degenerate bilinear form on 3-$\rho$-Lie algebra $\mathcal{B}_{\rho}$ to define a quadratic structure.
\begin{definition}
Let $\chi:\mathcal{B}_{\rho}\times \mathcal{B}_{\rho}\longrightarrow \mathbb{K}$ be a non-degenerate bilinear form on $\mathcal{B}_{\rho}$. Then $\chi$ is called a quadratic structure of $\mathcal{B}_{\rho}$ if, for any $f_1,f_2,f_3,f_4\in Hg(\mathcal{B}_{\rho})$, the following conditions are satisfied
\begin{enumerate}
\item[i.]
$\chi(f_1, f_2) = \rho(f_1,f_2) \chi(f_2,f_1)\quad$ ( $\rho$-symmetric),
\item[ii.]
$\chi([f_1,f_2,f_3]_\mathcal{B},f_4) = \chi(f_1, [f_2,f_3,f_4]_\mathcal{B})\quad$ (invariant).
\end{enumerate}
\end{definition}
Also $(\mathcal{B},[.,.,.]_\mathcal{B},\rho, \chi)$ is said to be a quadratic 3-$\rho$-Lie algebra if $\chi$ is a quadratic structure of $\mathcal{B}$. We use the symbol $\mathcal{B}_{\rho,\chi}$ to represent any quadratic 3-$\rho$-Lie algebra $(\mathcal{B},[.,.,.]_\mathcal{B},\rho, \chi)$.
\begin{definition}
Let ${\rm End(\mathcal{B}_{\rho})}$ be the endomorphism algebra of $\mathcal{B}_{\rho}$.	A vector subspace ${\rm Cent(\mathcal{B}_{\rho})}$ of ${\rm End(\mathcal{B}_{\rho})}$ is defined by
	\begin{align}
	{\rm Cent(\mathcal{B}_{\rho})} &= \{\psi\in {\rm End(\mathcal{B}_{\rho})} : \psi([f, g, h]_\mathcal{B}) = [\psi(f), g, h]_\mathcal{B}= \rho(\psi,f)[f, \psi(g), h]_\mathcal{B}=\rho(\psi, f+g) [f, g, \psi(h)]_\mathcal{B}\}.\label{159}
	\end{align}
	${\rm Cent(\mathcal{B}_{\rho})}$ is a subalgebra of ${\rm End(\mathcal{B}_{\rho})}$ and it is called the centroid of $\mathcal{B}_{\rho}$.	
\end{definition}
In the following proposition, we construct 3-$\rho$-Lie algebras from a 3-$\rho$-Lie algebra and an even element in its centroid.
\begin{proposition}\label{as}
	Let $\psi \in {\rm Cent(\mathcal{B}_{\rho})}$ be an even element. If we define two multiplications $[\cdot,\cdot,\cdot]^{\psi}$ and $[\cdot,\cdot,\cdot]^{\psi}_{\psi}$ on $\mathcal{B}$ by 
	$$[f,g,h]^{\psi}=[\psi(f), g, h]_\mathcal{B},\quad [f,g,h]^{\psi}_{\psi}=[\psi(f), \psi(g), \psi(h)]_\mathcal{B},\quad \forall f, g, h \in \mathcal{B}_{\rho},$$
	then we have two 3-$\rho$-Lie algebras
	$(\mathcal{B}, [\cdot,\cdot,\cdot]^{\psi},\rho)$ and $(\mathcal{B}, [\cdot,\cdot,\cdot]_{\psi}^
	{\psi}, \rho)$.
\end{proposition}
\begin{proof}
	The definition of the bracket $[.,.,.]^{\psi}$ and this fact that $\psi$ is an even element in ${\rm Cent(\mathcal{B}_{\rho})}$, lead to
	\begin{align*}
	&(a)~~[f_1, f_2, [g_1, g_2, g_3]^{\psi}]^{\psi}=\psi^2[f_1, f_2, [g_1, g_2, g_3]_\mathcal{B}]_\mathcal{B},\\
	&(b)~~[[f_1, f_2, g_1]^{\psi}, g_2, g_3]^{\psi}=\psi^2[[f_1, f_2, g_1]_\mathcal{B}, g_2, g_3]_\mathcal{B},\\
	&(c)~~\rho(f_1+f_2,g_1)[g_1,[f_1, f_2, g_2]^{\psi}, g_3]^{\psi}=\rho(f_1+f_2,g_1)\psi^2[g_1,[f_1, f_2, g_2]_\mathcal{B}, g_3]_\mathcal{B},\\
	&(d)~~\rho(f_1+f_2, g_1+g_2)[g_1, g_2, [f_1, f_2, g_3]^{\psi}]^{\psi}=\rho(f_1+f_2, g_1+g_2)\psi^2[g_1, g_2, [f_1, f_2, g_3]_\mathcal{B}]_\mathcal{B}.
	\end{align*}
	According to the above equalities and this fact that $\mathcal{B}_{\rho}$ is a 3-$\rho$-Lie algebra, it is easy to show that $[\cdot,\cdot,\cdot]^{\psi}$ is a 3-$\rho$-Lie algebra structure. Also, by the similar way, we can show that $(\mathcal{B}, [.,.,.]_{\psi}^{\psi}, \rho)$ is a 3-$\rho$-Lie algebra.
\end{proof}
\begin{proposition}\label{gh}
	Consider $\mathcal{B}_{\rho,\chi}$ and an even invertible $\chi$-symmetric element $\psi\in{\rm Cent(\mathcal{B}_{\rho})}$. If we define the even bilinear form $\chi_{\psi}$ by $\chi_{\psi}(f,g) = \chi(\psi(f), g)$, then we have two quadratic 3-$\rho$-Lie algebras $(\mathcal{B}, [.,.,.]^{\psi},\rho,\chi_{\psi})$ and $(\mathcal{B}, [.,.,.]_{\psi}^{\psi}, \rho, \chi_{\psi})$.
\end{proposition}
\begin{proof}
	$\chi_{\psi}$ is a non-degenerate and $\rho$-symmetric bilinear form, since $\psi$ is invertible and $\chi$ is a non-degenerate, $\rho$-symmetric bilinear form on $\mathcal{B}$. Then, it is enough to show that $\chi_{\psi}$ is invariant. For this, we have
	\begin{align*}
	\chi_{\psi}([f_1,f_2,f_3]^{\psi}, f_4)&=\chi(\psi[\psi(f_1),f_2,f_3]_A, f_4)=\chi([\psi(f_1),f_2,f_3], \psi(f_4))\\
	&=\chi(\psi(f_1),[f_2, f_3, \psi(f_4)]_A)=\chi_{\psi}(f_1,[f_2, f_3, f_4]^{\psi}).
	\end{align*}
	Also, we can show that $\chi_{\psi}([f_1,f_2,f_3]^{\psi}_{\psi}, f_4)=\chi_{\psi}(f_1,[f_2, f_3, f_4]^{\psi}_{\psi})$, where $f_1,f_2,f_3,f_4$ are homogeneous elements of $\mathcal{B}$.
\end{proof}
\begin{definition}
A 3-associative algebra is a pair $(g, \mu)$ consisting of a $G$-graded vector
space $g$ and an even 3-linear map $\mu : g\times g\times g\longrightarrow g$ (i.e. $\mu(g_a,g_b,g_c)\subset g_{a+b+c}$), such that for all $a,b,c,d,e\in g$
$$\mu(a,b, \mu(c,d,e)) = \mu(\mu(a,b,c), d,e).$$
If $\rho$ is a two-cycle on $G$ and the linear map $\mu$ is the $\rho$-symmetry with respect to the displacement of every two elements, the 3-associative algebra $(g, \mu)$ is called commutative and denoted by $g_{\rho,C}$.
\end{definition}
\begin{definition}
A quadratic 3-associative algebra is a 3-associative algebra $(g, \mu)$ equipped with a $\rho$-symmetric, invariant and non-degenerate bilinear form $\chi_g$. We denote a quadratic 3-associative algebra  $(g, \mu,\chi_g)$ and a commutative quadratic 3-associative algebra respectively by $g_{\chi_g}$ and  $g_{\rho,C,\chi_g}$.
\end{definition}
Note that if $A$ and $A^{\prime}$ are two $G$-graded vector spaces, then $A\otimes A^{\prime}$ is also a $G$-graded vector space such that $$A\otimes A^{\prime}=\oplus_{\gamma\in G}(A\otimes A^{\prime})_{\gamma},$$
where $(A\otimes A^{\prime})_{\gamma}=A_a\otimes A^{\prime}_{a^{\prime}}$ and so $f\otimes g\in(A\otimes A^{\prime})_{\gamma}$ has the degree $|f\otimes g|=\gamma=a+a^{\prime}$.\\

According to the above point, we have the following theorem, which gives us a quadratic 3-$\rho$-Lie algebra structure starting from a quadratic 3-$\rho$-Lie algebra and a quadratic commutative 3-associative algebra.
\begin{theorem}\label{a1}
Consider $\mathcal{B}_{\rho,\chi}$ and $g_{\rho,C,\chi_g}$. The tensor product $(E=\mathcal{B}\otimes g, [\cdot,\cdot,\cdot]_E, \rho, \chi_E)$ is a quadratic 3-$\rho$-Lie algebra, such that $[\cdot,\cdot,\cdot]_E$ and $\chi_E$ are given by
\begin{align}
[f\otimes a, g\otimes b, h\otimes c]_E &= \rho(a,g+h)\rho(b,h)[f,g,h]_\mathcal{B}\otimes\mu(a, b,c),\label{aaa}\\
\chi_E(f\otimes a, g\otimes b) &=\rho(a,g)\chi(f,g)\chi_g(a,b),\label{ccc}
\end{align}
for all $f,g,h\in Hg(\mathcal{B_{\rho,\chi}})$ and $ a, b,c \in Hg(g_{\rho,\chi})$.
\end{theorem}
\begin{proof}
Using \eqref{aaa} and the commutative property of product $\mu$, we get
\begin{align*}
&[f\otimes a, g\otimes b, [h_1\otimes c, h_2\otimes d, h_3\otimes e]_E]_E=\rho(a+b, h_1+h_2+h_3)\rho(c, h_2+h_3)\rho(d, h_3)\rho(a,g)\\
&\qquad \qquad \qquad\qquad \qquad \qquad \qquad\ \ \ \  \ \ \ \ \ \ \ \ \ \ [f,g,[h_1, h_2, h_3]_\mathcal{B}]_\mathcal{B}\otimes\mu(a,b,\mu(c,d,e)),\\
&[[f\otimes a, g\otimes b, h_1\otimes c]_E, h_2\otimes d, h_3\otimes e]_E=\rho(a+b, h_1+h_2+h_3)\rho(c, h_2+h_3)\rho(d, h_3)\rho(a,g)\\
&\qquad \qquad \qquad\qquad \qquad \qquad \qquad\ \ \ \  \ \ \ \ \ \ \ \ \ \ [[f,g,h_1]_\mathcal{B}, h_2, h_3]_\mathcal{B}\otimes\mu(a,b,\mu(c,d,e)),\\
&\rho(f+a+g+b, h_1+c)[ h_1\otimes c, [f\otimes a,g\otimes b, h_2\otimes d]_E, h_3\otimes e]_E=\rho(a+b, h_1+h_2+h_3)\rho(c, h_2+h_3)\\
&\qquad \qquad \qquad\qquad \qquad \qquad \qquad\ \ \ \  \ \ \ \ \ \rho(d, h_3)\rho(a,g)\rho(f+g, h_1)[h_1,[f,g, h_2]_\mathcal{B}, h_3]_\mathcal{B}\otimes\mu(a,b,\mu(c,d,e)),\\
&\rho(f+a+g+b, h_1+c+h_2+d)[h_1\otimes c,h_2\otimes d,[f\otimes a,g\otimes b, h_3\otimes e]_E]_E=\rho(a+b, h_1+h_2+h_3)\\
&\qquad \qquad \qquad\qquad \qquad\rho(c, h_2+h_3)\rho(d, h_3)\rho(a,g)\rho(f+g, h_1+h_2)[h_1,h_2,[f,g, h_3]_\mathcal{B}]_\mathcal{B}\otimes\mu(a,b,\mu(c,d,e)),\\
\end{align*}
where $f,g,h_1, h_2, h_3\in Hg(\mathcal{B}_{\rho,\chi})$ and $a,b,c, d, e\in Hg(g_{\rho,\chi,C})$. So $[.,.,.]_E$ is a 3-$\rho$-Lie algebra structure on $E$. Also, we have
$$\chi_E(f\otimes a, g\otimes b)=\rho(f+g,g+b)\chi_E(g\otimes b,f\otimes a),$$
and
\begin{align*}
\chi_E([f\otimes a, g\otimes b, h\otimes c]_E,h_1\otimes d) &= \rho(a,g+h)\rho(b,h)([f,g,h]_A\otimes\mu(a, b,c),h_1\otimes d)\\
&=\rho(a,g+h)\rho(b,h)\rho(a+b+c,h_2)\chi([f,g,h]_A,h_1)\chi_g(\mu(a,b,c), d)\\
&=\rho(a,g+h)\rho(b,h)\rho(a+b+c,h_2)\chi(f,[g,h_\mathcal{B},h_1]_\mathcal{B})\chi_g(a,\mu(b,c, d))\\
&=\rho(b,h)\rho(b+c,h_1)\chi_E(f\otimes a,[g,h,h_1]\otimes\mu(b,c,d))\\
&=\chi_E(f\otimes a, [g\otimes b, h\otimes c, h_1\otimes d]_E),
\end{align*}
which imply $\chi_E$ is a quadratic structure.
\end{proof}
\subsection{Ideal of 3-$\rho$-Lie algebras} In this part, we introduce the
definitions of subalgebra and ideal of $\mathcal{B}_{\rho}$ and give some properties related to ideals of $\mathcal{B}_{\rho}$.
\begin{definition}
A subalgebra of $\mathcal{B}_{\rho}$ is a sub-vector space $I\subseteq \mathcal{B}_{\rho}$ such that $[I,I,I]_\mathcal{B}\subseteq I$. $I$ is also called an ideal of $\mathcal{B}_{\rho}$ if $[I,\mathcal{B}_{\rho},\mathcal{B}_{\rho}]_\mathcal{B}\subseteq I$.
\end{definition}
\begin{definition}
Let $I$ be an ideal of $\mathcal{B}_{\rho,\chi}$.\\
(i)$~~$ $I$ is said to be non-degenerate if $\chi|_{I\times I}$ is non-degenerate.\\
(ii)$~~$  The orthogonal $I^{\perp}$ of $I$, with respect to $\chi$, is defined by
\begin{equation}
I^{\perp}=\{f\in \mathcal{B}_{\rho},~~ \chi(f,g)=0~~\forall g\in I\}.
\end{equation}
\end{definition}
\begin{lemma}
	Let $Z(\mathcal{B}_{\rho})$ be the center of $\mathcal{B}_{\rho}$ and defined by
	\begin{equation}
	Z(\mathcal{B}_{\rho})=\{f\in \mathcal{B}_{\rho}|~~[f,g,h]_{\mathcal{B}}=0~~\forall g,h\in \mathcal{B}_{\rho}\}.
	\end{equation}
	Then $Z(\mathcal{B}_{\rho})$ is an ideal of $\mathcal{B}_{\rho}$.
\end{lemma}
\begin{proof}
	It is clear that $[Z(\mathcal{B}_{\rho}), \mathcal{B}_{\rho}, \mathcal{B}_{\rho}]=\{0\}\subseteq Z(\mathcal{B}_{\rho})$.
\end{proof}
\begin{lemma}
Let $I$ be an ideal of $\mathcal{B}_{\rho,\chi}$. Then $I^{\perp}$ with respect to $\chi$ is an ideal of $\mathcal{B}_{\rho,\chi}$.
\end{lemma}
\begin{proof}
	At first, we show that $[I^{\perp}, \mathcal{B}_{\rho,\chi}, \mathcal{B}_{\rho,\chi}]\subseteq I^{\perp}$. For $f\in I^{\perp}$, $g,h\in Hg(\mathcal{B}_{\rho})$ and $s\in I$ we have
	$$\chi([f,g,h]_\mathcal{B}, s)=\chi(f,[g,h,s]_\mathcal{B})=\rho(g+h,s)\chi(f,[s,g,h])=0,$$
	since $f\in I^{\perp}$ and $I$ is an ideal of $\mathcal{B}_{\rho,\chi}$. 
\end{proof}
\textbf{Symplectic Structure:} In the following, we introduce the notions like symplectic structure $\omega$, metric $\varphi$, derivation $D$ and $\varphi$-antisymmetric derivation $D$ on $\mathcal{B}_{\rho}$ and show that a symplectic structure $\omega$ may be defined on $\mathcal{B}_{\rho}$ if and only if there exists a $\varphi$-antisymmetric invertible derivation $D$ of $\mathcal{B}_{\rho}$.
\begin{definition}\label{a15}
	A non-degenerate and $\rho$-skew-symmetric bilinear form $\omega\in\wedge^2\mathcal{B}^*$ is called a symplectic structure on $\mathcal{B}_{\rho}$ if 
	\begin{align*}
	&\omega([f_1, f_2, f_3]_\mathcal{B}, f_4)-\rho(f_1, f_2+f_3+f_4)\omega([f_2, f_3, f_4]_\mathcal{B}, f_1)\\
	&\ \ \ +\rho(f_1+f_2,f_3+f_4)\omega([f_3, f_4, f_1]_\mathcal{B}, f_2)-\rho(f_1+f_2+f_3, f_4)\omega([f_4, f_1, f_2]_\mathcal{B},f_3)=0.
	\end{align*}
\end{definition}
We will say that $(\mathcal{B},[.,.,.]_\mathcal{B},\rho, \chi, \omega)$ is a quadratic symplectic 3-$\rho$-Lie algebra if $(\mathcal{B},\chi)$ is quadratic 
and $(\mathcal{B},\omega)$ is symplectic.\\
In the next example, we are going to check the $4$-dimensional 3-$\rho$-Lie algebra, which it's classical case is given in \cite{FF}.
\begin{example}\label{Ex.123}
	Let us consider $\mathcal{B}$ as a $G=\mathbb{Z}_2^3$-graded vector space with the basis $\{l_1, l_2, l_3, l_4\}$. Then $\mathcal{B}=\mathcal{B}_{(1,0,0)}\oplus\mathcal{B}_{(0,1,0)}\oplus\mathcal{B}_{(0,0,1)}\oplus\mathcal{B}_{(1,1,1)}$ with
	$$\quad l_1\in\mathcal{B}_{(1,0,0)},\quad l_2\in\mathcal{B}_{(0,1,0)},\quad l_3\in\mathcal{B}_{(0,0,1)},\quad l_4\in\mathcal{B}_{(1,1,1)}.$$
	Also, $\rho:G\times G\longrightarrow \mathbb{C}^*$ is defined by the matrix
	$$
	[\rho(a,b)]=
	\begin{bmatrix}
	1	& -1 & -1 & 1 \\ 
	-1	& 1 &  -1&  1\\ 
	-1	& -1 & 1 &  1\\ 
	1	& 1 &  1& 1
	\end{bmatrix},
	$$
	where $(a,b)\in \{(1,0,0), (0,1,0), (0,0,1), (1,1,1)\}\times  \{(1,0,0), (0,1,0), (0,0,1), (1,1,1)\}$. The 3-$\rho$-Lie algebra bracket associate to the basis $\{l_1, l_2, l_3, l_4\}$ is given by
	\begin{align*}
	&[l_1, l_2, l_3]_\mathcal{B}=l_4,\quad [l_1, l_2, l_4]_\mathcal{B}=l_3,\\
	&[l_2, l_3, l_4]_\mathcal{B}=l_1,\quad [l_1, l_3,l_4]_\mathcal{B}=l_2.
	\end{align*}
	It is clear that any non-degenerate $\rho$-skew symmetric bilinear form is a symplectic structure on $\mathcal{B}$. For example
	\begin{align*}
	&\omega_1 = l^*_3\wedge l^*_1+ l^*_4\wedge l^*_2,\quad \omega_2 = l^*_2\wedge l^*_1+ l^*_4\wedge l^*_3,\quad \omega_3 = l^*_2\wedge l^*_1+ l^*_3\wedge l^*_4,\\
	&\omega_4 = l^*_1\wedge l^*_2+ l^*_4\wedge l^*_3,\quad
	\omega_5 = l^*_1\wedge l^*_2+ l^*_3\wedge l^*_4,\quad 
	\omega_6 = l^*_1\wedge l^*_3+ l^*_2\wedge l^*_4,
	\end{align*}
	where $\{l_1^*,l_2^*,l_3^*,l_4^*\}$ is the dual basis of  $\{l_1, l_2, l_3, l_4\}$.
\end{example}
\begin{definition}
A derivation on $\mathcal{B}_{\rho}$ is a linear map $D: \mathcal{B}_{\rho}\longrightarrow \mathcal{B}_{\rho}$ such that
\begin{align*}
&D[f,g,h]_\mathcal{B}= [D(f), g, h] +\rho(D, f) [f, D(g), h] + \rho(D, f+g)[f,g,D(h)],
\end{align*}
for any $f,g,h\in Hg(\mathcal{B}_{\rho})$. Let us denote by $Der(\mathcal{B}_{\rho})$ the space of all derivations of $\mathcal{B}_{\rho}$.
\end{definition}
\begin{definition}
	A metric on $\rho$-Lie algebra $\mathcal{B}$ is a non-degenerate $\rho$-symmetric bilinear form $\varphi:\mathcal{B}\times \mathcal{B}\rightarrow \mathbb{K}$ such that 
	$$\varphi([f,g],h)=-\rho(f,g)\varphi(g,[f,h]).$$
	Also, non-degenerate $\rho$-symmetric bilinear form $\varphi$ is said to be a metric on $\mathcal{B}_{\rho}$ if
	\begin{equation}
	\varphi([f,g,h],z) +\rho(f+g,h) \varphi(h, [f,g,z]) = 0,\quad \forall f,g,h,z\in Hg(\mathcal{B}_{\rho}).
	\end{equation}
	In this case, $(\mathcal{B}_{\rho},\varphi)$ is called a metric 3-$\rho$-Lie algebra and we denote it by $\mathcal{B}_{\rho,\varphi}$.
	\end{definition}
\begin{definition}
Consider $\mathcal{B}_{\rho,\chi}$ ($(\mathcal{B}_{\rho,\varphi})$ ) and let $D$ be an endomorphism or a derivation of $\mathcal{B}_{\rho}$. $D$ is called $\chi$-antisymmetric ($\varphi$-antisymmetric) if $$\chi(D(f),g)=-\rho(D,f)\chi(f,D(g))\quad (\varphi(D(f),g)=-\rho(D,f)\varphi(f,D(g))),$$
for all $f,g\in Hg(\mathcal{B}_{\rho})$. Denote by ${\rm End_{\chi}(\mathcal{B}_{\rho})}$ (${\rm End_{\varphi}(\mathcal{B}_{\rho})}$) or ${\rm Der_{\chi}(\mathcal{B}_{\rho})}$ (${\rm Der_{\varphi}(\mathcal{B}_{\rho})}$) the space of $\chi$-antisymmetric ($\varphi$-antisymmetric) endomorphism or derivation of $\mathcal{B}_{\rho}$.
\end{definition}
\begin{proposition}
Consider the 3-$\rho$-Lie algebra $\mathcal{B}_{\rho}$. The following assertions are equivalent:
\begin{enumerate}
\item[i.]
 The existence of a symplectic structure on $\mathcal{B}_{\rho}$,
\item[ii.]
The existence of an invertible derivation $D\in {\rm Der_{\varphi}(\mathcal{B}_{\rho})}$.
\end{enumerate}
\end{proposition}
\begin{proof}
	At first, we assume that there is an invertible derivation $D$ on $\mathcal{B}_{\rho}$ such that 
	$\varphi(D(f), g)=-\rho(D,f)\varphi(f,D(g))$. For $f,g\in \mathcal{B}_{\rho}$, define $\omega(f,g)=\varphi(D(f),g)$. Thus $\omega$ is a non-degenerate and $\rho$-skew symmetric bilinear form. Now, we show that it is a symplectic structure. For homogeneous elements $f,g,h,z\in \mathcal{B}_{\rho}$, we deduce 
	\begin{align*}
	&\omega([f_1,f_2,f_3]_\mathcal{B},f_4)-\rho(f_1, f_2+f_3+f_4)\omega([f_2,f_3,f_4]_\mathcal{B},f_1)+\rho(f_1+f_2,f_3+f_4)\omega([f_3,f_4,f_1]_\mathcal{B},f_2)\\
	&\ \ \ -\rho(f_1+f_2+f_3,f_4)\omega([f_4,f_1,f_2]_\mathcal{B},f_3)\\
	&=\varphi(D[f_1,f_2,f_3]_\mathcal{B},f_4)-\rho(f_1,f_2+f_3+f_4)\varphi(D[f_2,f_3,f_4]_\mathcal{B},f_1)\\
	&\ \ \ +\rho(f_1+f_2,f_3+f_4)\varphi(D[f_3,f_4,f_1]_\mathcal{B},f_2)-\rho(f_1+f_2+f_3,f_4)\varphi(D[f_4,f_1,f_2]_\mathcal{B},f_3)\\
	&=\varphi(D[f_1,f_2,f_3]_\mathcal{B},f_4)-\varphi(D[f_1,f_2,f_3]_\mathcal{B},f_4)=0.
	\end{align*}
	Conversely is similar.
\end{proof}
Consider the quadratic 3-$\rho$-Lie algebra $\mathcal{B}_{\rho,\chi}$. For any symmetric bilinear form $\chi^{\prime}$ on $\mathcal{B}_{\rho}$, there is an associated map $D:\mathcal{B}_{\rho}\longrightarrow \mathcal{B}_{\rho}$ satisfying
\begin{equation}\label{ssss}
\chi^{\prime}(f,g) = \chi(D(f),g),\quad \forall f,g\in \mathcal{B}_{\rho}.
\end{equation}
Since $\chi$ and $\chi^{\prime}$ are symmetric, then $D$ is symmetric with respect to $\chi$ , i.e., $$\chi(D(f),g) =\rho(D,f)\chi(f,D(g)),\quad \forall f,g\in Hg(\mathcal{B}_{\rho}).$$
\begin{lemma}\label{lem 25}
Let $\chi^{\prime}$ be defined by \eqref{ssss} on $\mathcal{B}_{\rho}$. Then
\begin{itemize}
\item[i.]
$\chi^{\prime}$ is invariant if and only if $D$ satisfies
\begin{equation}\label{eq.12}
D([f,g,h]_\mathcal{B}) = [D(f),g,h]_\mathcal{B} = \rho(D,f)[f,D(g),h]_\mathcal{B}=\rho(D,f+g)[f,g,D(h)]_\mathcal{B}.
\end{equation}
\item[ii.]
$\chi^{\prime}$ is non-degenerate if and only if $D$ is invertible.
\end{itemize}
\end{lemma}
\begin{proof}
(i) Assuming $f,g,h,s\in \mathcal{B}_{\rho}$, we have $$\chi^{\prime}([f,g,h]_\mathcal{B},s) = \chi(D([f,g,h]_\mathcal{B}),s)~~\text{
 and}~~ \chi^{\prime}(f,[g,h,s]_\mathcal{B}) = \chi(D(f),[g,h,s]_\mathcal{B}).$$ 
Since $\chi$ is invariant, so
 $\chi^{\prime}(f,[g,h,s]_\mathcal{B})= \chi([D(f),g,h]_\mathcal{B},s)$. On the other hand, since $\chi$ is non-degenerate, therefore $\chi^{\prime}$ is invariant if and only if $D([f,g,h]_\mathcal{B}) = [D(f),g,h]_\mathcal{B}$. 
Also, since 3-Lie algebra structure $[.,.,.]_\mathcal{B}$ is anticommutative with respect to displacement of every two elements, then $D([f,g,h]_\mathcal{B}) = \rho(D,f)[f,D(g),h]_\mathcal{B}=\rho(D,f+g)[f,g,D(h)]_\mathcal{B}$.\\
(ii) We assume that $\chi^{\prime}$ is non-degenerate and $f\in \mathcal{B}_{\rho}$ such that $D(f) = 0$. Then
$\chi(D(f),\mathcal{B}_{\rho}) = 0$ and so $\chi^{\prime}(f,\mathcal{B}_{\rho}) = 0$. Therefore $f=0$ and $D$ is
invertible. Conversely, assume that $D$ is invertible and $\chi^{\prime}(f,\mathcal{B}_{\rho}) = 0$. So $\chi(D(f),\mathcal{B}_{\rho}) = 0$. Since $\chi$ is non-degenerate, thus $D(f) = 0$. $D$ is invertible, therefore $f = 0$ and $\chi^{\prime}$ is non-degenerate.
\end{proof}
\begin{definition}
A $\rho$-symmetric map $D:\mathcal{B}_{\rho}\longrightarrow \mathcal{B}_{\rho}$ satisfying \eqref{eq.12} is called a centromorphism of $\mathcal{B}_{\rho}$. The space of centromorphisms of $\mathcal{B}_{\rho}$ is denote by $\mathscr{C}(\mathcal{B}_{\rho})$.
\end{definition}
\begin{proposition}\label{pro 26}
Let $\delta\in {\rm Der_{\chi}(\mathcal{B}_{\rho})}$ and $D\in\mathscr{C}(\mathcal{B}_{\rho})$ such that $D\circ \delta=\delta\circ D$. Then $D\circ \delta$ is also a $\chi$-antisymmetric derivation of $\mathcal{B}_{\rho}$.
\end{proposition}
\begin{proof}
We have 
\begin{align*}
D\circ \delta[f,g,h]&= D[\delta(f),g,h]_{\mathcal{B}}+\rho(\delta,f)D[f,\delta(g),h]_{\mathcal{B}} +\rho(\delta,f+g) [f,g,\delta(h)]_{\mathcal{B}}\\&=[D\circ\delta(f),g,h]_{\mathcal{B}}+\rho(D+\delta,f)D[f,D\circ\delta(g),h]_{\mathcal{B}}+\rho(\delta,f+g) [f,g,D\circ\delta(h)]_{\mathcal{B}},
\end{align*}
and
$B(D\circ\delta(f),g)=\rho(D+\delta,f)B(f,\delta\circ D(g)) = -\rho(D+\delta,f)B(f,D\circ \delta(g))$ for all $f,g\in Hg(\mathcal{B}_{\rho})$. So we deduce that $D\circ\delta \in {\rm Der_{\chi}(\mathcal{B}_{\rho})}$.
\end{proof}
\section{Representation of 3-$\rho$-Lie algebras}
In this section we extend the representation theory of 3-Lie algebras introduced in \cite{CLY} and \cite{SK} to color Lie version. We introduce a representation of 3-$\rho$-Lie algebras and discuss the cases of adjoint and coadjoint representations for 3-$\rho$-Lie algebras. \\

Consider $\mathcal{L}=\wedge^2\mathcal{B}_{\rho}$ and call it the fundamental set. Defining 
\begin{equation}\label{123}
[(f_1,f_2), (g_1,g_2)]_{\mathcal{L}}=([f_1, f_2, g_1],g_2)+\rho(f_1+f_2,g_1)(g_1,[f_1,f_2,g_2]),
\end{equation}
we construct a $\rho$-Lie algebra structure on $\mathcal{L}$.
\begin{example}
Let $(\mathcal{B},[.,.,.]_\mathcal{B},\rho,\varphi)$ be a metric 3-$\rho$-Lie algebra. Then $(\mathcal{L},[.,.]_{\mathcal{L}},\varphi_{\mathcal{L}})$ is a metric $\rho$-Lie algebra, where 
$$\varphi_{\mathcal{L}}((f_1,f_2), (g_1,g_2))=\varphi([f_1,f_2,g_1]_\mathcal{B},g_2).$$
\end{example}
\begin{definition}\label{a5}
The pair $(V,\mu)$ is called a representation of $\mathcal{B}_{\rho}$ if the following conditions are satisfied	
	\begin{align}
	\mu[(f_1,f_2),(g_1,g_2)]_{_\mathcal{L}}&=\mu(f_1,f_2)\mu(g_1,g_2)-\rho(f_1+f_2,g_1+g_2)\mu(g_1,g_2)\mu(f_1,f_2),\label{f}\\
	\mu([g_1,g_2,g_3]_\mathcal{B},f)&=\mu(g_1,g_2)\mu(g_3,f)+\rho(g_1,g_2+g_3)\mu(g_2,g_3)\mu(g_1,f)\label{g}\\
	&\ \ \  +\rho(g_1+g_2,g_3)\mu(g_3,g_1)\mu(g_2,f),\nonumber\\
	\mu(g,[f_1,f_2,f_3]_\mathcal{B})&=\rho(g,f_1+f_2)\mu(f_1,f_2)\mu(g,f_3)+\rho(g,f_2+f_3)\rho(f_1,f_2+f_3)\mu(f_2,f_3)\mu(g,f_1)\label{h}\\
	&\ \ \  +\rho(g,f_1+f_3)\rho(f_1+f_2,f_3)\mu(f_3,f_1)\mu(g,f_2),\nonumber
	\end{align}
where $V$ is a $G$-graded vector space and $\mu$ is a linear map from $\mathcal{L}=\wedge^2\mathcal{B}_{\rho}$ to ${\rm gl(V)}$.
\end{definition}
One of the important examples of the 3-$\rho$-Lie algebra representation is the adjoint representation, which is a bilinear map $ad:\mathcal{B}_{\rho}\times \mathcal{B}_{\rho}\longrightarrow {\rm gl(\mathcal{B}_{\rho})}$ defined by $ad(f_1,f_2)(f_3)=[f_1,f_2,f_3]_\mathcal{B}$, that is a representation of $\mathcal{B}_{\rho}$ on itself. 
\begin{lemma}\label{aa5}
For the representation $(V,\mu)$ of $\mathcal{B}_{\rho}$, we have
	\begin{align*}
	0&=\rho(f_1+f_2,g_1)\mu(g_1, [f_1,f_2,g_2]_\mathcal{B}) +\mu([f_1,f_2,g_1]_\mathcal{B}, g_2)\\
	&\quad +\rho(f_1+f_2,g_1+g_2)\rho(g_1+g_2,f_1)\mu(f_1, [g_1, g_2, f_2]_\mathcal{B})+\rho(f_1+f_2, g_1+g_2)\mu([g_1, g_2, f_1]_\mathcal{B}, f_2).
	\end{align*}
\end{lemma}
\begin{proof}
	Using Definition \ref{a5}, we obtain 
	\begin{align*}
	&\rho(f_1+f_2,g_1)\mu(g_1, [f_1,f_2,g_2]_\mathcal{B}) +\mu([f_1,f_2,g_1]_\mathcal{B}, g_2)\\
	&\quad +\rho(f_1+f_2, g_1+g_2)\mu([g_1, g_2, f_1]_\mathcal{B}, f_2)\\
	&\quad +\rho(f_1+f_2,g_1+g_2)\rho(g_1+g_2,f_1)\mu(f_1, [g_1, g_2, f_2]_\mathcal{B})\\
	&=\mu(f_1, f_2)\mu(g_1, g_2)-\rho(f_1+f_2, g_1+g_2)\mu(g_1, g_2)\mu(f_1, f_2)\\
	&\quad +\rho(f_1+f_2, g_1+g_2)\mu(g_1, g_2)\mu(f_1, f_2)-\mu(f_1, f_2)\mu(g_1, g_2)=0.
	\end{align*}
\end{proof}
Consider the pair $(\mathcal{B}_{\rho}, V)$, where $V$ is a $G$-graded vector space. The direct sum 
$\mathcal{B}\oplus V$ is also graded and $(\mathcal{B}\oplus V)_a=\mathcal{B}_a\oplus V_a$. Furthermore, a homogeneous element of $\mathcal{B}\oplus V$ has the form $f+v$ such that $f\in \mathcal{B}_{\rho}$ and $v\in V$ and $|f+v|=|f|=|v|$.\\

In the following, we discuss some properties of 3-$\rho$-Lie algebras representations. Actually, the following theorem examines the relationship between the 3-$\rho$-Lie algebra representation and the existence of 3-$\rho$-Lie algebra structure on the graded vector space $\mathcal{B}\oplus V$:
\begin{proposition}\label{zahra}
The pair $(V,\mu)$ is a representation of $\mathcal{B}_{\rho}$  if and only if  $(\mathcal{B}\oplus V, [.,.,.]_{\mathcal{B}\oplus V})$ is a 3-$\rho$-Lie algebra, where the structure $[.,.,.]_{\mathcal{B}\oplus V}$ is given by
	$$[f_1+v_1, f_2+v_2, f_3+v_3]^{\mu}_{\mathcal{B}\oplus V} = [f_1, f_2, f_3]_\mathcal{B} +\mu(f_1, f_2)v_3 +\rho(f_1, f_2+f_3)\mu(f_2,f_3)v_1+\rho(f_1+f_2,f_3)\mu(f_3,f_1)v_2,$$
	for homogeneous elements $f_1, f_2, f_3\in \mathcal{B}_{\rho}$ and $v_1, v_2, v_3\in V$.
\end{proposition}
\begin{proof}
	The proof of this proposition follows by a direct computation and the definition of the representation.
\end{proof}
In the following, we describe the dual representations and coadjoint representations of 3-$\rho$-Lie algebras such as the adjoint representations. At first, we recall the dual space of a graded vector space and explore the graded space of direct sum of a graded vector space and it's dual.\\

Let $\mathcal{B}=\oplus_{a\in G}\mathcal{B}_a$ be a $G$-graded space. Then $\mathcal{B}^{\star}$, the dual space of $\mathcal{B}$, is also $G$-graded space and $$\mathcal{B}^{\star}_a=\{\alpha\in \mathcal{B}^{\star}| \alpha(f)=0,~~\forall~~f:~~|f|\neq -a\}.$$ Furthermore, the direct sum 
$\mathcal{B}\oplus \mathcal{B}^{\star}$ is $G$-graded, where
$$\mathcal{B}\oplus \mathcal{B}^{\star}=\oplus_{a\in G}(\mathcal{B}\oplus \mathcal{B}^{\star})_a=\oplus_{a\in G}(\mathcal{B}_a\oplus \mathcal{B}^{\star}_a),$$
A homogeneous element of $\mathcal{B}\oplus \mathcal{B}^{\star}$ has the form $f+\alpha$ such that $f\in \mathcal{B}$ and $\alpha\in \mathcal{B}^{\star}$ and $|f+\alpha|=|f|=|\alpha|$.
\begin{proposition}\label{a13}
Let $(V,\mu)$ be a $\rho$-skew-symmetric representation of $\mathcal{B}_{\rho}$ and let $V^{\star}$ be the dual vector space of $V$. Defining the linear map $\mu^*:\mathcal{B}_{\rho}\times \mathcal{B}_{\rho}\longrightarrow {\rm gl(V^{\star})}$ by $\mu^*(f_1,f_2)(\varrho)=-\rho(f_1+f_2,\varrho)\varrho\circ\mu(f_1,f_2)$, where $f_1,f_2\in \mathcal{B},~~\varrho\in V^{\star}$,  $(V^{\star},\mu^*)$ defines a representation of $\mathcal{B}_{\rho}$, which is called the dual representation. 
\end{proposition}
\begin{proof}
	By invoking Definition \ref{a5} and Lemma \ref{aa5}, we can prove this proposition by direct computations.
\end{proof}
\begin{remark}
An example of the dual representation is the linear map $ad^{\star}:\mathcal{B}_{\rho}\times \mathcal{B}_{\rho}\longrightarrow {\rm gl(\mathcal{B}^{\star}_{\rho})}$ defined by $$ad^{\star}(f_1,f_2)(\varrho)(f_3)=-\rho(f_1+f_2,\varrho)\varrho[f_1,f_2,f_3]_\mathcal{B}=-\rho(f_1+f_2,\varrho)\varrho(ad(f_1,f_2)(f_3),$$ for $f_1,f_2,f_3\in Hg(\mathcal{B}_{\rho})$ and $\varrho\in \mathcal{B}^{\star}_{\rho}$, which is a representation of $\mathcal{B}$ on $\mathcal{B}^*_{\rho}$ ($ad^{\star}$ is called coadjoint representation). Since $ad^*$ is a representation of $\mathcal{B}_{\rho}$, then by Proposition \ref{zahra}, $\mathcal{B}\oplus \mathcal{B}^*$ is a 3-$\rho$-Lie algebra by the following structure
\begin{align}\label{dd}
[f_1+\alpha_1, f_2+\alpha_2, f_3+\alpha_3]^{ad^*}_{\mathcal{B}\oplus \mathcal{B}^*}&=[f_1,f_2,f_3]_\mathcal{B}+ad^*(f_1,f_2)\alpha_3\\
&\ \ \ +\rho(f_1, f_2+f_3)ad^*(f_2,f_3)\alpha_1+\rho(f_1+f_2,f_3)ad^*(f_3,f_1)\alpha_2.\nonumber
\end{align}
Considering $\varphi:\mathcal{B}\oplus \mathcal{B}^*\times \mathcal{B}\oplus \mathcal{B}^*\longrightarrow\mathbb{K}$ by $\varphi(f+\alpha,g+\beta)=\alpha(g)+\rho(f,g)\beta(f)$, we have a metric 3-$\rho$-Lie algebra $(\mathcal{B}\oplus \mathcal{B}^*, [.,.,.]^{ad^*}_{\mathcal{B}\oplus \mathcal{B}^*},\rho,\varphi)$.
\end{remark}
\section{3-pre-$\rho$-Lie algebra}
In \cite{CLY} the notion of 3-pre-Lie algebras was introduced as a generalization of pre-Lie algebras. In this section, we introduce the notion of 3-pre-$\rho$-Lie algebras and investigate their properties. Also, the 3-pre-$\rho$-Lie algebras representations and the phase spaces of 3-pre-$\rho$-Lie algebras will be examined.
\begin{definition}\label{a6}
	A triple $(\mathcal{B},\{.,.,.\},\rho)$ is called a 3-pre-$\rho$-Lie algebra if the following equalities hold
	\begin{align}
	|\{f_1, f_2, f_3\}|&=|f_1| +|f_2| +|f_3|,\nonumber\\
	\{f_1, f_2, f_3\}&=-\rho(f_1,f_2)\{f_2, f_1, f_3\},\\
	\{f_1,f_2, \{g_1, g_2,g_3\}\}&=\{[f_1,f_2,g_1]_c, g_2, g_3\} +\rho(f_1+f_2,g_1)\{g_1, [f_1,f_2,g_2]_c, g_3\}\label{a2}\\
	&\ \ \  +\rho(f_1+f_2,g_1+g_2)\{g_1, g_2, \{f_1,f_2,g_3\}\},\nonumber\\
	\{[f_1,f_2,f_3]_c, g_1, g_2\}&=\{f_1, f_2, \{f_3, g_1, g_2\}\} +\rho(f_1,f_2+f_3)\{f_2, f_3, \{f_1, g_1,g_2\}\}\label{a3}\\
	&\ \ \  +\rho(f_1+f_2,f_3)\{f_3, f_1, \{f_2,g_1,g_2\}\},\nonumber
	\end{align}
	where $\mathcal{B}$ is a $G$-graded vector space, $\{.,.,.\}:\otimes^3 \mathcal{B}\longrightarrow \mathcal{B}$ is a trilinear map, $\rho$ is a two-cycle and 
	\begin{align*}
	[f_1, f_2, f_3]_c=\{f_1, f_2, f_3\}+\rho(f_1, f_2+f_3)\{f_2, f_3, f_1\}+\rho(f_1+f_2,f_3)\{f_3, f_1, f_2\}.
	\end{align*}
	To simplify notation, we write $\mathcal{B}_{\{\}}$ instead of $(\mathcal{B}, \{.,.,.\}, \rho)$.
\end{definition}
\begin{proposition}\label{a9}
$([.,.,.]_c,\rho)$ give a 3-$\rho$-Lie algebra structure, when $(\mathcal{B}, \{.,.,.\}, \rho)$ is a 3-pre-$\rho$-Lie algebra. $(\mathcal{B},[.,.,.]_c,\rho)$ is called the sub-adjacent 3-$\rho$-Lie algebra of $\mathcal{B}$ and is denoted by $\mathcal{B}^c_{\rho}$. 
\end{proposition}
\begin{proof}
	Using the definition of $[.,.,.]_c$, \eqref{a2} and \eqref{a3}, we have
	\begin{align*}
	&[f_1, f_2, [f_3, f_4, f_5]_c]_c-[[f_1, f_2, f_3]_c, f_4, f_5]_c-\rho(f_1+f_2,f_3)[f_3, [f_1, f_2, f_4]_c,f_5]_c\\
	&\quad -\rho(f_1+f_2,g_1+g_2)[f_3, f_4, [f_1, f_2, f_5]_c]_c=0.
	\end{align*}
\end{proof}
Note that, $\mathcal{B}_{\{\}}$ is called the compatible 3-$\rho$-pre-Lie algebra structure on the 3-$\rho$-Lie algebra $\mathcal{B}^c_{\rho}$.\\
The following lemma gives a representation of $\mathcal{B}^c_{\rho}$ by the left multiplication:
\begin{lemma}\label{111}
Defining the left multiplication $L:\wedge^2\mathcal{B}_{\{\}}\longrightarrow gl(\mathcal{B}_{\{\}})$ by $L(f,g)h=\{f,g,h\}$, $(\mathcal{B},L)$ is a representation of $\mathcal{B}^c_{\rho}$.
\end{lemma}
\begin{proof}
	By Definition \ref{a5}, we have
	\begin{align*}
	L[(f_1,f_2), (g_1,g_2)]_{\mathcal{L}}(h)&=L([f_1,f_2,g_1]_c,g_2)(h)+\rho(f_1+f_2,g_1)L(g_1,[f_1,f_2,g_2]_c)(h)\\
	&=\{[f_1,f_2,g_1]_c,g_2,h\}+\rho(f_1+f_2,g_1)\{g_1,[f_1,f_2,g_2]_c,h\}\\
	&=\{f_1,f_2,[g_1,g_2,h]_c\}-\rho(f_1+f_2,g_1+g_2)\{g_1,g_2,\{f_1,f_2,h\}\}\\
	&=L(f_1,f_2)L(g_1,g_2)(h)-\rho(f_1+f_2,g_1+g_2)L(g_1,g_2)L(f_1,f_2)(h).
	\end{align*}
	The other ones prove in a similar way.
\end{proof}
\textbf{Representation of 3-$\rho$-pre-Lie algebras:}
Here, we are going to represent the representation theory of 3-$\rho$-pre-Lie algebras. Also, the dual representation will be given.
\begin{definition}\label{a7}
 A pair $(\mu, \tilde{\mu})$ is called a representation of  $\mathcal{B}_{\{\}}$ on $G$-graded vector space $V$ if for all homogeneous elements $f_1, f_2, f_3, f_4\in Hg(\mathcal{B})$, the following equalities hold
	\begin{align*}
	\tilde{\mu}(f_1, \{f_2, f_3, f_4\})&=\rho(f_1, f_2+f_3)\mu(f_2, f_3)\tilde{\mu}(f_1, f_4)+\rho(f_1, f_3+f_4)\rho(f_2, f_3+f_4)\tilde{\mu}(f_3, f_4)\tilde{\mu}(f_1, f_2)\\
	&\ \ \ -\rho(f_1, f_2+f_4)\rho(f_3, f_4)\tilde{\mu}(f_2, f_4)\tilde{\mu}(f_1, f_3)+\rho(f_1+f_2,f_3+f_4)\tilde{\mu}(f_3, f_4)\mu(f_1, f_2)\\
	&\ \ \ +\rho(f_1, f_2+f_3+f_4)\rho(f_3, f_4)\tilde{\mu}(f_2, f_4)\tilde{\mu}(f_3, f_1)\\
	&\ \ \ -\rho(f_1, f_2+f_4)\rho(f_3, f_4)\tilde{\mu}(f_2, f_4)\mu(f_1, f_3)\\
	&\ \ \ -\rho(f_2, f_3+f_4)\rho(f_1, f_2+f_3+f_4)\tilde{\mu}(f_3, f_4)\tilde{\mu}(f_2,f_1),\\
	&\\
	\mu(f_1, f_2)\tilde{\mu}(f_3, f_4)&=\rho(f_1+f_2, f_3+f_4)\tilde{\mu}(f_3, f_4)\mu(f_1, f_2)-\rho(f_1+f_2, f_3+f_4)\rho(f_1, f_2)\tilde{\mu}(f_3, f_4)\tilde{\mu}(f_2, f_1)\\
	&\ \ \ +\rho(f_1+f_2, f_3+f_4)\tilde{\mu}(f_3, f_4)\tilde{\mu}(f_1, f_2)+\tilde{\mu}([f_1, f_2, f_3]_c, f_4)\\
	&\ \ \ +\rho(f_1+f_2, f_3)\tilde{\mu}(f_3, \{f_1, f_2, f_4\}),\\
	&\\
	\tilde{\mu}([f_1, f_2, f_3]_c, f_4)&=\mu(f_1, f_2)\tilde{\mu}(f_3, f_4)+\rho(f_1, f_2+f_3)\mu(f_2, f_3)\tilde{\mu}(f_1, f_4)+\rho(f_1+f_2,f_3)\mu(f_3, f_1)\tilde{\mu}(f_2, f_4),\\
	&\\
	\tilde{\mu}(f_3, f_4)\mu(f_1, f_2)&=\rho(f_1, f_2)\tilde{\mu}(f_3, f_4)\tilde{\mu}(f_2, f_1)-\tilde{\mu}(f_3, f_4)\tilde{\mu}(f_1, f_2)+\rho(f_3+f_4, f_1+f_2)\mu(f_1, f_2)\tilde{\mu}(f_3, f_4)\\
	&\ \ \ -\rho(f_3+f_4, f_1+f_2)\rho(f_1, f_2)\tilde{\mu}(f_2, \{f_1, f_3, f_4\})+\rho(f_3+f_4, f_1+f_2)\tilde{\mu}(f_1, \{f_2, f_3, f_4\}),
	\end{align*}
where $\mu:\wedge^2\mathcal{B}^c_{\rho}\longrightarrow gl(V)$ is a representation of the 3-$\rho$-Lie algebra $\mathcal{B}^c_{\rho}$ and $\tilde{\mu}:\otimes^2\mathcal{B}\longrightarrow gl(V)$ is a bilinear map.
\end{definition}
In the following lemma, we construct a representation of $\mathcal{B}_{\{\}}$ by the left and right multiplication. The left and right multiplication are two bilinear maps of the form $L,R:\otimes^2\mathcal{B}_{\{\}}\longrightarrow gl(\mathcal{B}_{\{\}})$ which are defined by $L(f,g)h=\{f,g,h\}$ and $R(f,g)h=\rho(f+g,h)\{h,f,g\}$, respectively.
\begin{lemma}
	Let $\mathcal{B}_{\{\}}$ be a 3-pre-$\rho$-Lie algebra. Then the pair $(L, R)$ is a representation of $\mathcal{B}_{\{\}}$ on itself (this representation is called the regular representation).
\end{lemma}
\begin{proof}
	For sample, let us examine one of the properties of the representation. For this, we have
	\begin{align*}
	&L(f_1, f_2)R(f_3, f_4)f_5+\rho(f_1, f_2+f_3)L(f_2, f_3)R(f_1, f_4)f_5+\rho(f_1+f_2, f_3)L(f_3, f_1)R(f_2, f_4)f_5\\
	&=\rho(f_3+f_4,f_5)\{f_1, f_2, \{f_5, f_3, f_4\}\}+\rho(f_1, f_2+f_3)\rho(f_1+f_4, f_5)\{f_2, f_3, \{f_5, f_1, f_4\}\}\\
	&\ \ \ +\rho(f_1+f_2, f_3)\rho(f_2+f_4,f_5)\{f_3, f_1, \{f_5, f_2, f_4\}\}\\
	&=-\rho(f_4, f_5)\{[f_1, f_2, f_3]_c, f_5, f_4\}=\rho(f_1+f_2+f_3+f_4, f_5)\{f_5, [f_1, f_2, f_3]_c, f_4\}\\
	&=
	R([f_1, f_2, f_3]_c, f_4)f_5.
	\end{align*}
\end{proof}
Consider $(\mathcal{B}, \{.,.,.\},\rho)$ and $(V,\mu)$ as a 3-pre-$\rho$-Lie algebra and a representation of $\mathcal{B}^c_{\rho}$, respectively. Then $(V,\mu, 0)$ is a representation of $\mathcal{B}_{\{\}}$.
\begin{proposition}\label{a8}
Let $(V,\mu, \tilde{\mu} )$ be a representation of $\mathcal{B}_{\{\}}$. Then $(\mathcal{B}\oplus V, \{.,.,.\}_{\mu, \tilde{\mu}}, \rho)$ is a 3-pre-$\rho$-Lie algebra, where the operation $\{.,.,.\}_{\mu, \tilde{\mu}}:\otimes^3(\mathcal{B}\oplus V)\longrightarrow \mathcal{B}\oplus V$ is defined by
	\begin{align}\label{a10}
	\{f_1+v_1, f_2+v_2, f_3+v_3\}_{\mu, \tilde{\mu}}=\{f_1, f_2, f_3\}+\mu(f_1, f_2)v_3+\rho(f_1, f_2+f_3)\tilde{\mu}(f_2, f_3)v_1-\rho(f_2, f_3)\tilde{\mu}(f_1, f_3)v_2,
	\end{align}
	and $f_1, f_2, f_3\in Hg(\mathcal{B})$ and $v_1, v_2, v_3\in Hg(V)$, 
\end{proposition}
\begin{proof}
	Using Definitions \ref{a6} and \ref{a7} and a direct computation, we get the result.
\end{proof}
Let $V$ be a graded vector space. Define the operator $\zeta:\otimes^2V\longrightarrow\otimes^2V$ by $ \zeta(f_1\otimes f_2)=f_2\otimes f_1$, where $f_1\otimes f_2\in\otimes^2V$.
\begin{proposition}\label{a14}
	Let $(V, \mu, \tilde{\mu})$ be a representation of $\mathcal{B}_{\{\}}$. If we define the bilinear map $\nu:\wedge^2\mathcal{B}\longrightarrow gl(V)$ by $\nu(f_1, f_2)=(\mu-\rho(f_1, f_2)\tilde{\mu}\zeta+\tilde{\mu})(f_1, f_2)$ for all $f_1, f_2\in Hg(\mathcal{B})$, then $(V,\nu)$ is a representation of $\mathcal{B}^c_{\rho}$.
\end{proposition}
\begin{proof}
	Proposition \ref{a8} implies that we have the 3-pre-$\rho$-Lie algebra $(\mathcal{B}\oplus V, \{.,.,.\}_{\mu, \tilde{\mu}}, \rho)$. Also, Proposition \ref{a9} implies that $(\mathcal{B}^c\oplus V,[.,.,.]_c,\rho)$ is the sub-adjacent 3-$\rho$-Lie algebra, where
	\begin{align*}
	[f_1+v_1, f_2+v_2, f_3+v_3]^{\mu,\tilde{\mu}}_c&=\{f_1+v_1, f_2+v_2, f_3+v_3\}_{\mu,\tilde{\mu}}+\rho(f_1, f_2+f_3)\{f_2+v_2, f_3+v_3, f_1+v_1\}_{\mu,\tilde{\mu}}\\
	&\ \ \ +\rho(f_1+f_2, f_3)\{f_3+v_1, f_1+v_1, f_2+v_2\}_{\mu,\tilde{\mu}}.
	\end{align*}
	By using \eqref{a10}, we get
	\begin{align}
	[f_1+v_1, f_2+v_2, f_3+v_3]^{\mu,\tilde{\mu}}_c&=[f_1, f_2, f_3]_c+(\mu-\rho(f_1,f_2)\tilde{\mu}\xi+\tilde{\mu})(f_1,f_2)v_3\label{a12}\\
	&\ \ \ +\rho(f_1, f_2+f_3)(\mu-\rho(f_2,f_3)\tilde{\mu}\xi+\tilde{\mu})(f_2,f_3)v_1\nonumber\\
	&\ \ \ +\rho(f_1+f_2,f_3)(\mu-\rho(f_3,f_1)\tilde{\mu}\xi+\tilde{\mu})(f_3,f_1)v_2\nonumber\\
	&=[f_1, f_2, f_3]_c+\nu(f_1,f_2)v_3+\rho(f_1, f_2+f_3)\nu(f_2,f_3)v_1\nonumber\\
	&\quad +\rho(f_1+f_2,f_3)\nu(f_3, f_1)v_2.\nonumber
	\end{align}
	Finally, by Proposition \ref{zahra}, we deduce that $(V,\nu)$ is a representation of $\mathcal{B}^c_{\rho}$.
\end{proof}
According to the above proposition, we can deduce that if $(\mathcal{B},\{.,.,.\},\rho)$ is a 3-pre-$\rho$-Lie algebra with the representation $(V,\mu, \tilde{\mu})$, then the 3-pre-$\rho$-Lie algebras $(\mathcal{B}\oplus V,\{.,.,.\}_{\mu, \tilde{\mu}},\rho)$ and  $(\mathcal{B}\oplus V, \{.,.,.\}_{\nu,0},\rho)$ have the same sub-adjacent 3-$\rho$-Lie algebra given by \eqref{a12}.

In the next proposition, we find the dual of the representation $(V,\mu, \tilde{\mu})$ of 3-pre-$\rho$-Lie algebra $(\mathcal{B},\{.,.,.\},\rho)$. Let us define $\nu^*(f_1, f_2)=\mu^*-\rho(f_1,f_2)\tilde{\mu}^*\zeta+\tilde{\mu}^*(f_1,f_2)$. 
\begin{proposition}
Consider $\mathcal{B}_{\{\}}$ equipped with a $\rho$-skew-symmetric representation $(V,\mu,\tilde{\mu})$. Then $(V^*,\nu^*, -\tilde{\mu}^*)$ is a representation of $\mathcal{B}_{\{\}}$, which is called the dual representation of the representation $(V,\mu, \tilde{\mu})$.
\end{proposition}
\begin{proof}
	By Propositions \ref{a13}, \ref{a14} and Definition \ref{a7}, we get the result.
\end{proof}
If $\mathcal{B}_{\{\}}$ is a 3-pre-$\rho$-Lie algebra with the $\rho$-skew-symmetric representation $(V,\mu, \tilde{\mu})$, then the 3-pre-$\rho$-Lie algebras $(\mathcal{B}\oplus V^*,\{.,.,.\}_{\mu^*,0},\rho)$ and $(\mathcal{B}\oplus V^*,\{.,.,.\}_{\nu^*, -\tilde{\mu}^*},\rho)$ have the same sub-adjacent 3-$\rho$-Lie algebra $(\mathcal{B}^c\oplus V^*,[.,.,.]_c^{\mu^*},\rho)$.\\

\textbf{$\rho$-$\mathcal{O}$-operator.} In the following, we define the notion of  $\rho$-$\mathcal{O}$-operator associated to the representation $(V,\mu)$ of $\mathcal{B}_{\rho}$, then we construct a 3-$\rho$-pre-Lie algebra structure on the representation $V$ by a $\rho$-$\mathcal{O}$-operator $T$ and bilinear map $\mu$. In the next, we investigate the relationship between compatible 3-$\rho$-pre-Lie algebra structure and an invertible $\rho$-$\mathcal{O}$-operator of $\mathcal{B}_{\rho}$ (for the classical case refer to \cite{CLY}).
\begin{definition}
Let $\mathcal{B}$ be a 3-$\rho$-Lie algebra with the representation $(V,\mu)$.
An even linear operator $T : V \longrightarrow \mathcal{B}$ is called a $\rho$-$\mathcal{O}$-operator associated to $(V,\mu)$ if $T$ satisfies
\begin{align*}
[Tx, Ty, Tz]_\mathcal{B} = T (\mu(Tx, Ty)z +\rho(x, y+z)(\mu(Ty, Tz)x + \rho(x+y, z)\mu(Tz, Tx)y),\quad \forall x,y,z \in V. 
\end{align*}
\end{definition}
\begin{lemma}
Consider the following three assumptions
\begin{enumerate}
\item[i.]
 $\mathcal{B}_{\rho}$ is a 3-$\rho$-Lie algebra,
 \item[ii.]
 $(V, \mu)$ is a  $\rho$-skew-symmetric representation of $\mathcal{B}_{\rho}$,
 \item[iii.]
 $T$ is a $\rho$-$\mathcal{O}$-operator associated to $(V, \mu)$.
\end{enumerate} If we define a new multiplication on $V$ by
	\begin{align}\label{457}
	\{x,y,z\}_V = \mu(Tx, Ty)z,\quad \forall x,y,z\in V,
	\end{align}
 then $(V, \{.,.,.\}_V)$ is a 3-pre-$\rho$-Lie algebra.
\end{lemma}
\begin{proof}
	Suppose that $x,y,z\in Hg(V)$. It is easy to see that 
	\begin{align*}
	\{x,y,z\}=-\rho(x,y)\{y,x,z\},\qquad T[x,y,z]_c=[Tx,Ty,Tz].
	\end{align*}
	Also, for $x_1, x_2, x_3, x_4, x_5\in Hg(V)$, we have
	\begin{align}
	\{x_1, x_2, \{x_3, x_4, x_5\}\}&=\mu(Tx_1, Tx_2)\mu(Tx_3, Tx_4)x_5,\label{z1}\\
	\{[x_1, x_2, x_3]_c, x_4, x_5\}&=\mu([Tx_1, Tx_2,Tx_3], Tx_4)x_5,\label{z2}\\
	\rho(x_1+x_2, x_3)\{x_3, [x_1, x_2, x_4]_c, x_5\}&=\rho(x_1+x_2, x_3)\mu(Tx_3, [Tx_1, Tx_2, Tx_4])x_5,\label{z3}\\
	\rho(x_1+x_2, x_3+x_4)\{x_3, x_4,\{x_1,x_2,x_5\}\}&=\rho(x_1+x_2, x_3+x_4)\mu(Tx_3, Tx_4)\mu(Tx_1, Tx_2)x_5.\label{z4}
	\end{align}
	By Definition \ref{a5} and (\ref{z1})--(\ref{z4}), we get the result.
\end{proof}
\begin{corollary}
By the above proposition, we deduce that $(V, [\cdot,\cdot,\cdot]_C)$ is a 3-$\rho$-Lie algebra as the sub-adjacent of the 3-$\rho$-pre-Lie algebra $(V,\{\cdot,\cdot,\cdot\})$. Furthermore, $T(V) = \{T v|v \in V \}\subset A $ is a subalgebra of $\mathcal{B}_{\rho}$ and there is an induced 3-$\rho$-pre-Lie algebra structure on $T (V )$ given by
\begin{equation}\label{456}
\{Tu, T v, Tw\}_{T(V )} := T \{u, v, w\},\quad \forall u, v, w \in V. 
\end{equation}
\end{corollary}
\begin{lemma}\label{459}
Consider 3-$\rho$-Lie algebra $\mathcal{B}_{\rho}$. Then the following assertions are equivalent
\begin{itemize}
\item[i.]
existence of a compatible 3-pre-$\rho$-Lie algebra structure,
\item[ii.]
existence of an invertible $\rho$-$\mathcal{O}$-operator.
\end{itemize}
\end{lemma}
\begin{proof}
 Let $T$ be an invertible $\rho$-$\mathcal{O}$-operator of $\mathcal{B}_{\rho}$ associated to $(V,\mu)$. BY \eqref{457}, \eqref{456} and this fact that $T$ is a $\rho$-$\mathcal{O}$-operator, we can easily deduce that  $(A=T(V), \{\cdot,\cdot,\cdot\}_{T(V)})$ is a compatible 3-$\rho$-pre-Lie algebra. Conversely, the identity map id is an $\rho$-$\mathcal{O}$-operator of the sub-adjacent 3-$\rho$-Lie algebra of a 3-$\rho$-pre-Lie algebra associated to the representation $(\mathcal{B},L)$.
\end{proof}
\subsection{Phase spaces of 3-$\rho$-Lie algebras}
In this section, the notion of a phase space of a 3-$\rho$-Lie algebra will be introduced and we show that a 3-$\rho$-Lie algebra has a phase space if and only if it is sub-adjacent to a 3-$\rho$-pre-Lie algebra. 
\begin{proposition}\label{ass}
Consider the symplectic 3-$\rho$-Lie algebra $(\mathcal{B}, [.,.,.]_\mathcal{B},\rho,\omega)$. Then the following structure defines a compatible 3-pre-$\rho$-Lie algebra structure on $\mathcal{B}$
\begin{equation}\label{458}
\omega(\{f,g,h\},s) = -\rho(f+g,h)\omega(h, [f,g,s]_\mathcal{B}),~~\forall f,g,h,s\in \mathcal{B}.
\end{equation}
\end{proposition}
\begin{proof}
Let us define a linear map $T:\mathcal{B}^*\longrightarrow \mathcal{B}$ by  $\omega(f, g)=(T^{-1}(f))(g)$, or equivalently, $\omega(T\alpha, g) =\alpha(g)$ for all $f,g \in Hg(\mathcal{B})$ and $\alpha\in\mathcal{B}^*$. Since $\omega$ is a symplectic structure, we deduce that $T$ is an invertible $\rho$-$\mathcal{O}$-operator associated to $(\mathcal{B}^*, Ad^*)$. Also, by \eqref{457} and \eqref{456}, there there exists a compatible 3-$\rho$-pre-Lie algebra on $B$ given by $\{f,g,h\} = T (ad^*(f,g)T^{-1}(h))$. So, \eqref{458} holds.
\end{proof}
\begin{definition}
$(\mathcal{B}\oplus \mathcal{B}^*, [.,.,.]_{\mathcal{B}\oplus \mathcal{B}^*}, \omega)$ is called a phase space of $\mathcal{B}_{\rho}$, if
\begin{enumerate}
\item[i.]
$[.,.,.]_{\mathcal{B}\oplus \mathcal{B}^*}$ is a 3-$\rho$-Lie algebra structure on the vector space $\mathcal{B}\oplus \mathcal{B}^*$,
\item[ii.]
 $\omega$ is a natural non-degenerate $\rho$-skew-symmetric bilinear form on $\mathcal{B}\oplus \mathcal{B}^*$ given by 
 \begin{equation}\label{112}
 \omega(f+\alpha, g+\beta)=\alpha(g)-\rho(f,g)\beta(f),
 \end{equation}
 such that $(\mathcal{B}\oplus \mathcal{B}^*, [.,.,.]_{\mathcal{B}\oplus \mathcal{B}^*},\rho, \omega)$ is a symplectic 3-$\rho$-Lie algebra,
\item[iii.]
 $(\mathcal{B}, [.,.,.]_\mathcal{B},\rho)$ and $(\mathcal{B}^*, [.,.,.]_{\mathcal{B}^*},\rho)$ are 3-$\rho$-Lie subalgebras of $(\mathcal{B}\oplus \mathcal{B}^*, [.,.,.]_{\mathcal{B}\oplus \mathcal{B}^*},\rho)$.
\end{enumerate}
\end{definition}
\begin{theorem}\label{th.2020}
	A 3-$\rho$-Lie algebra has a phase space if and only if it is sub-adjacent to a 3-pre-$\rho$-Lie algebra.
\end{theorem}
\begin{proof}
	Let $\mathcal{B}_{\{\}}$ be a 3-pre-$\rho$-Lie algebra. By Propositions \ref{zahra}, \ref{a13}, \ref{ass} and Lemma \ref{111} we have the 3-$\rho$-Lie algebra $(\mathcal{B}^c\oplus \mathcal{B}^*, [.,.,.]^{L^*}_{\mathcal{B}^c\oplus \mathcal{B}^*},\rho)$. Moreover, for all $f_1, f_2, f_3, f_4\in Hg(\mathcal{B})$ and $\alpha_1, \alpha_2, \alpha_3, \alpha_4\in Hg(\mathcal{B}^*)$, by (\ref{dd}) with respect to $L^*$, we have
	\begin{align*}
	&\omega([f_1+\alpha_1,f_2+\alpha_2, f_3+\alpha_3 ]^{L^*}_{\mathcal{B}^c\oplus \mathcal{B}^*}, f_4+\alpha_4)\\
	&=-\rho(f_1+f_2+f_3, f_4)\alpha_4\{f_1, f_2, f_3\}-\rho(f_1, f_2+f_3)\rho(f_1+f_2+f_3, f_4)\alpha_4\{f_2, f_3, f_1\}\\
	&\ \ \ -\rho(f_1+f_2,f_3)\rho(f_1+f_2+f_3, f_4)\alpha_4\{f_3, f_1, f_2\}-\rho(f_1+f_2,f_3)\alpha_3\{f_1, f_2, f_4\}\\
	&\ \ \ -\alpha_1\{f_2, f_3, f_4\}-\rho(f_1+f_2,f_3)\rho(f_3+f_1, f_2)\alpha_2\{f_3, f_1, f_4\},\\
	&\\
	&\rho(f_1, f_2+f_3+f_4)\omega([f_2+\alpha_2,f_3+\alpha_3, f_4+\alpha_4 ]^{L^*}_{\mathcal{B}^c\oplus \mathcal{B}^*}, f_1+\alpha_1)\\
	&=-\alpha_1\{f_2, f_3, f_4\}-\rho(f_2, f_3+f_4)\alpha_1\{f_3, f_4, f_2\}-\rho(f_2+f_3,f_4)\alpha_1\{f_4, f_2, f_3\}\\
	&\ \ \ -\rho(f_1,f_2+f_3+f_4)\rho(f_2+f_3,f_4)\alpha_4\{f_2, f_3, f_1\}-\rho(f_1,f_2+f_3+f_4)\alpha_2\{f_3, f_4, f_1\}\\
	&\ \ \ -\rho(f_1,f_2+f_3+f_4)\rho(f_2+f_3,f_4)\rho(f_4+f_2, f_3)\alpha_3\{f_4, f_2, f_1\},\\
	&\\
	&\rho(f_1+f_2,f_3+f_4)\omega([f_3+\alpha_3,f_4+\alpha_4, f_1+\alpha_1 ]^{L^*}_{\mathcal{B}^c\oplus \mathcal{B}^*}, f_2+\alpha_2)\\
	&=-\rho(f_1,f_3+f_4+f_2)\alpha_2\{f_3, f_4, f_1\}\\
	&\ \ \ -\rho(f_1,f_4+f_1)\rho(f_3,f_4)\alpha_2\{f_4, f_1, f_3\}\\
	&\ \ \ -\rho(f_2,f_3+f_4)\rho(f_1+f_3+f_4,f_2)\alpha_2\{f_1, f_3, f_4\}\\
	&\ \ \ -\rho(f_2,f_3+f_4)\alpha_1\{f_3, f_4, f_2\}-\rho(f_1+f_2,f_3+f_4)\alpha_3\{f_4, f_1, f_2\}\\
	&\ \ \ -\rho(f_2,f_3+f_4)\rho(f_3+f_1, f_4)\alpha_4\{f_1, f_3, f_2\},\\
	&\\
	&\rho(f_1+f_2+f_3,f_4)\omega([f_4+\alpha_4,f_1+\alpha_1, f_2+\alpha_2 ]^{L^*}_{\mathcal{B}^c\oplus \mathcal{B}^*}, f_3+\alpha_3)\\
	&=-\rho(f_1+f_2,f_3+f_4)\alpha_3\{f_4, f_1, f_2\}-\rho(f_1+f_2,f_3)\alpha_3\{f_1, f_2, f_4\}\\
	&\ \ \ -\rho(f_1+f_2,f_3+f_4)\rho(f_4+f_1,f_2)\alpha_3\{f_2, f_4, f_1\}-\rho(f_1+f_3,f_4)\rho(f_1,f_2)\alpha_2\{f_4, f_1, f_3\}\\
	&\ \ \ -\rho(f_3,f_4)\rho(f_1+f_2,f_4)\alpha_4\{f_1, f_2, f_3\}-\rho(f_3,f_4)\alpha_1\{f_2, f_4, f_3\}.
	\end{align*}
	By the above relations, we deduce that $\omega$ is a symplectic structure on $\mathcal{B}^c\oplus\mathcal{B}^*$. Moreover, $\mathcal{B}^c$ and $\mathcal{B}^*$ are the subalgebra and the abelian subalgebra of $\mathcal{B}^c\oplus\mathcal{B}^*$, respectively. Thus, the symplectic 3-$\rho$-Lie algebra $(\mathcal{B}^c\oplus\mathcal{B}^*,[.,.,.]^{L^*}_{\mathcal{B}^c\oplus\mathcal{B}^*},\rho, \omega)$ is a phase space of $\mathcal{B}^c_{\rho}$. Conversely, let $(\mathcal{B}\oplus \mathcal{B}^*, [.,.,.]_{\mathcal{B}\oplus \mathcal{B}^*}, \rho,\omega)$ be a phase space of $\mathcal{B}_{\rho}$. 
	By Proposition \ref{ass}, there is a compatible 3-pre-$\rho$-Lie algebra structure $\{.,.,.\}$ on $\mathcal{B}\oplus \mathcal{B}^*$. Thus
	$$\omega(\{f,g,h\}, z) = -\rho(f+g,h)\omega(h, [f,g,z]_{\mathcal{B}\oplus \mathcal{B}^*}) = -\rho(f+g,h)\omega(h, [f,g,z]_\mathcal{B}) = 0 \quad \forall f,g,h,z\in Hg(\mathcal{B}).$$
	The last equality holds, since $\mathcal{B}_{\rho}$ is a subalgebra of $(\mathcal{B}\oplus \mathcal{B}^*, [.,.,.]_{\mathcal{B}\oplus \mathcal{B}^*},\rho)$.
	Therefore, $\{f,g,h\}\in \mathcal{B}_{\rho}$ that says $(\mathcal{B},\{.,.,.\}|_{\mathcal{B}},\rho)$ is a subalgebra of the 3-pre-$\rho$-Lie algebra $(\mathcal{B}\oplus \mathcal{B}^*, \{.,.,.\},\rho)$ and so its sub-adjacent 3-$\rho$-Lie algebra $\mathcal{B}^c_{\rho}$ is the 3-$\rho$-Lie algebra $\mathcal{B}_{\rho}$.
\end{proof}
\begin{corollary}
Let $(\mathcal{B}\oplus \mathcal{B}^*, [.,.,.]_{\mathcal{B}\oplus \mathcal{B}^*},\rho, \omega)$ be a phase space of $\mathcal{B}_{\rho}$ and $(\mathcal{B}\oplus \mathcal{B}^*, \{.,.,.\},\rho)$ be the associated 3-pre-$\rho$-Lie algebra. Then $(\mathcal{B},\{.,.,.\}|_{\mathcal{B}},\rho)$ and $(\mathcal{B}^*,\{.,.,.\}|_{\mathcal{B}^*},\rho)$ are subalgebras of the 3-pre-$\rho$-Lie algebra $(\mathcal{B}\oplus \mathcal{B}^*,\{.,.,.\},\rho)$.
\end{corollary}
\begin{lemma}
Let $(V,\mu)$ be a representation of $\mathcal{B}_{\rho}$ and $(\mathcal{B}\oplus \mathcal{B}^*, [.,.,.]^{\mu^*}_{\mathcal{B}\oplus \mathcal{B}^*},\rho, \omega)$ be a phase space of it. Then there exists a 3-pre-$\rho$-Lie algebra structure on $\mathcal{B}$ given by
$$\{f,g,h\}= \mu(f,g)h,\quad \forall f,g,h\in \mathcal{B}.$$
\end{lemma}
\begin{proof}
Assume that $f,g,h\in Hg(\mathcal{B})$ and $\alpha\in Hg(\mathcal{B}^*)$. Then, we have
\begin{align*}
 \rho(f+g+h,\alpha)\alpha(\{f,g,h\})&=-\omega(\{f,g,h\},\alpha)=\rho(f+g,h)\omega(h,[f,g,\alpha]_{\mathcal{B}\oplus \mathcal{B}^*})=\rho(f+g,h)\omega(h,\mu^*(f,g)\alpha)\\
 &=-\rho(f+g,h)\rho(h,f+g+\alpha)\mu^*(f,g)\alpha(h)= \rho(f+g+h,\alpha)\alpha\mu(f,g)h.
\end{align*}
\end{proof}
%-----------------------------------------------------------------------------------------------
\section{quadratic 3-hom-$\rho$-lie algebras}
In this section, we study 3-Hom-$\rho$-lie algebras, which 3-$\rho$-lie algebras are special case of these Lie algebras with respect to the map $Id$. The constructions of Section 2, such as quadratic, symplectic structures, 3-associative Lie algebras and etc, will be investigated in this section to color Hom-Lie case. In this section, the representation theory of 3-Hom-$\rho$-lie algebras is also included.
\begin{definition}
Multiple $(\mathcal{B},[.,.]_{\mathcal{B}},\rho,\phi)$ is called a 2-Hom-$\rho$-Lie algebra (or simply a Hom-$\rho$-Lie algebra) if 
\begin{enumerate}
\item[i.] 
$\mathcal{B}$ is a $G$-graded vector space, where $G$ is an abelian group,
\item[ii.]
$\rho$ is a two-cycle,
\item[iii.]
$\phi:\mathcal{B}\longrightarrow \mathcal{B}$ is an even linear map,
\item[iv.]
$[.,.]_{\mathcal{B}}$ is a a bilinear map on $\mathcal{B}$, satisfying the following conditions
\begin{align*}
&\bullet |[f,g]_{\mathcal{B}}|= | f|+| g|,\\
&\bullet [f, g]_{\mathcal{B}} = -\rho(f,g)[g, f]_{\mathcal{B}},\\
&\bullet \rho(h,f)[\phi(f), [g, h]_{\mathcal{B}}]_{\mathcal{B}}+\rho(g,h)[\phi(h), [f, g]_{\mathcal{B}}]_{\mathcal{B}}+\rho(f,g)[\phi(g), [h, f]_{\mathcal{B}}]_{\mathcal{B}} = 0\quad \forall f,g,h\in \mathcal{B}.
\end{align*} 
\end{enumerate}
The third condition is equivalent to 
$$[\phi(f), [g, h]_{\mathcal{B}}]_{\mathcal{B}}=[[f, g]_{\mathcal{B}},\phi(h)]_{\mathcal{B}}+\rho(f,g)[\phi(g), [f, h]_{\mathcal{B}}]_{\mathcal{B}}.$$
\end{definition}
\begin{definition}
Let
\begin{enumerate}
\item[i.] 
$\mathcal{B}$ be a $G$-graded vector space,
\item[ii.]
$\phi:\mathcal{B}\longrightarrow \mathcal{B}$ be an even linear map,
\item[iii.]
$\rho$ be a two-cycle,
\item[iv.]
$[.,.,.]_{\mathcal{B}}$ be a trilinear map on $\mathcal{B}$ satisfying
\begin{align*} 
&|[f_1,f_2,f_3]_{\mathcal{B}}|=|f_1|+ |f_2| +|f_3|,\\
&[\phi(f_1),\phi(f_2),[g_1,g_2,g_3]_{\mathcal{B}}]_{\mathcal{B}}=[[f_1,f_2,g_1]_{\mathcal{B}},\phi(g_2),\phi(g_3)]_{\mathcal{B}}+\rho(f_1+f_2,g_1)[\phi(g_1),[f_1,f_2,g_2]_{\mathcal{B}},\phi(g_3)]_{\mathcal{B}}\\
&\ \ \ \ \ \ \ \ \ \ \ \ \ \ \ \ \ \ \ \ \ \ \ \ \ \ \ \ \ \ \ \ \ \ \ \ +\rho(f_1+f_2, g_1+g_2)[\phi(g_1),\phi(g_2),[f_1,f_2,g_3]_{\mathcal{B}}]_{\mathcal{B}},\\
&[f_1,f_2,f_3]_{\mathcal{B}}=-\rho(f_1,f_2)[f_2,f_1,f_3]_{\mathcal{B}}=-\rho(f_2,f_3)[f_1,f_3,f_2]_{\mathcal{B}},
\end{align*}
\end{enumerate}
for any $f_1,f_2,f_3,g_1,g_2\in Hg(\mathcal{B})$. Then $(\mathcal{B},[\cdot,\cdot,\cdot]_{\mathcal{B}},\rho,\phi)$ is called a 3-$\rho$-Lie algebra. Let us denote the 3-Hom-$\rho$-Lie algebra $(\mathcal{B},[\cdot,\cdot,\cdot]_{\mathcal{B}},\rho,\phi)$ by $\mathcal{B}_{\rho,\phi}$.
\end{definition}
\begin{definition}
Let $(\mathcal{B},[\cdot,\cdot,\cdot]_\mathcal{B},\phi)$ and $(\mathcal{A},[\cdot,\cdot,\cdot]_{\mathcal{A}},\psi)$ be two 3-Hom-$\rho$-Lie algebras. A linear map $\alpha:\mathcal{B}\longrightarrow \mathcal{A}$ is said to be a morphism of 3-Hom-$\rho$-Lie algebras if for all $f,g,h\in \mathcal{B}$
$$\alpha[f,g,h]_\mathcal{B}=[\alpha(f),\alpha(g),\alpha(h)]_{\mathcal{A}},$$
 and 
$$\alpha\circ\phi=\psi\circ\alpha.$$
\end{definition}
\begin{definition}
$\mathcal{B}_{\rho,\phi}$ is said to be
\begin{enumerate}
\item
 multiplicative if $\phi$ is a Lie algebra morphism, i.e. for any $f,g,h\in \mathcal{B}$,
 $\phi[f,g,h]_{\mathcal{B}}=[\phi(f),\phi(g),\phi(h)]_{\mathcal{B}}.$
 \item
regular if $\phi$ is an automorphism for $[\cdot,\cdot,\cdot]_\mathcal{B}$,
\item
involutive if $\phi^2 = Id_\mathcal{B}$.
\end{enumerate}
\end{definition}
In the following theorem, a new 3-Hom-$\rho$-Lie algebra structure will be constructed by combining a 3-Hom-$\rho$-Lie algebra structure and an even 3-$\rho$-Lie algebra endomorphism.
\begin{theorem}\label{Th.124}
Consider $\mathcal{B}_{\rho,\phi}$ and let $\beta:\mathcal{B}_{\rho,\phi}\longrightarrow \mathcal{B}_{\rho,\phi}$ be 
an even 3-$\rho$-Lie algebra endomorphism. Then $(\mathcal{B},[\cdot,\cdot,\cdot]_{\beta}=\beta\circ [\cdot,\cdot,\cdot]_\mathcal{B},\beta\circ \phi)$ is a 3-Hom-$\rho$-Lie algebra.
\end{theorem}
\begin{proof}
It is easy to see that $[\cdot,\cdot,\cdot]_{\beta}$ is $\rho$-skew symmetric with respect to the displacement of every two elements. Now, we check the $\rho$-fundamental identity:
\begin{align*}
&(a)~~[\beta\circ\phi(f_1), \beta\circ\phi(f_2),[f_3,f_4,f_5]_{\beta}]_{\beta}=[\beta\circ\phi(f_1), \beta\circ\phi(f_2),\beta\circ[f_3,f_4,f_5]_\mathcal{B}]_{\beta}\\
&\ \ \ \ \ \ \ \ \ \ \ \ \ \ \ \ \ \ \ \ \ \ \ \ \ \ \ \ \ \ \ \ \ \ \ \ \ \ \ \ \ \ \ \ =\beta^2[\phi(f_1), \phi(f_2),[f_3,f_4,f_5]_\mathcal{B}]_\mathcal{B},\\
&(b)~~[[f_1,f_2,f_3]_{\beta},\beta\circ\phi(f_4), \beta\circ\phi(f_5)]_{\beta}=\beta^2[[f_1,f_2,f_3]_\mathcal{B},\phi(f_4), \phi(f_5)]_\mathcal{B},\\
&(c)~~\rho(f_1+f_2,f_3)[\beta\circ\phi(f_3),[f_1,f_2,f_4]_{\beta},\beta\circ\phi(f_5)]_{\beta}=\rho(f_1+f_2,f_3)\beta^2[\phi(f_3),[f_1,f_2,f_4]_\mathcal{B},\phi(f_5)]_\mathcal{B}\\
&(d)~~\rho(f_1+f_2,f_3+f_4)[\beta\circ\phi(f_3),\beta\circ\phi(f_4),[f_1,f_2,f_5]_{\beta}]_{\beta}=\rho(f_1+f_2,f_3+f_4)\beta^2[\phi(f_3),\phi(f_4),,[f_1,f_2,f_5]_\mathcal{B}]_\mathcal{B}.
\end{align*}
According to the above four items and this fact that $\mathcal{B}$ is a 3-Hom-$\rho$-Lie algebra, the proof is complete.
\end{proof}
\begin{example}
Consider $\mathcal{B}_{\rho,\phi}$ and let $\phi$ be a 3-$\rho$-Lie algebra
morphism. Then $(\mathcal{B},[\cdot,\cdot,\cdot]_{\phi}=\phi\circ [\cdot,\cdot,\cdot]_\mathcal{B},\rho,\phi)$ is a multiplicative 3-Hom-$\rho$-Lie algebra.
\end{example}
\begin{definition}
A non-degenerate bilinear form $\chi$ on $\mathcal{B}_{\rho,\phi}$ satisfying
	\begin{enumerate}
		\item[i.]
		$\chi(f_1, f_2) = \rho(f_1,f_2) \chi(f_2,f_1)\quad \forall f_1,f_2\in Hg(\mathcal{B})$~~ ( $\rho$-symmetric),
		\item[ii.]
		$\chi([f_1,f_2,f_3]_\mathcal{B},f_4) = \chi(f_1, [f_2,f_3,f_4]_\mathcal{B})\quad\forall f_1,f_2,f_3,f_4\in \mathcal{B}$ (invariant),
		\item[iii.]
		$\chi(\phi(f_1), f_2)=\chi(f_1,\phi(f_2))\quad$ ($\phi$ is $\chi$-symmetric).
	\end{enumerate}
	is called a quadratic structure on $\mathcal{B}_{\rho,\phi}$. A quadratic 3-Hom-$\rho$-Lie algebra is a 3-Hom-$\rho$-Lie algebra equipped with a quadratic structure. We denote a quadratic 3-Hom-$\rho$-Lie algebra by $\mathcal{B}_{\rho,\phi,\chi}$. Also, We recover quadratic 3-$\rho$-Lie algebras when $\phi=id_{\mathcal{B}}$.
\end{definition}
Consider the subalgebra ${\rm Cent(\mathcal{B}_{\rho})}$ of ${\rm End(\mathcal{B}_{\rho})}$ defined by \ref{159}, we have the following:
\begin{proposition}
	Let $\psi$ be an even element in ${\rm Cent(\mathcal{B}_{\rho,\phi})}$. For any $f,g,h\in Hg(\mathcal{B})$, we define two following structures
	$$[f,g,h]^{\psi}=[\psi(f), g, h]_\mathcal{B},\quad [f,g,h]^{\psi}_{\psi}=[\psi(f), \psi(g). \psi(h)]_\mathcal{B}.$$
	Then, we have the following 3-Hom-$\rho$-Lie algebras
	\begin{align*}
	&(\mathcal{B}, [.,.,.],\rho,\psi\circ\phi),\quad (\mathcal{B}, [.,.,.]^{\psi},\rho,\phi),\quad (\mathcal{B}, [.,.,.]^{\psi},\rho,\psi\circ\phi),\\
	&(\mathcal{B}, [.,.,.]_{\psi}^{\psi}, \rho, \phi), \quad (\mathcal{B}, [.,.,.]_{\psi}^{\psi}, \rho, \psi\circ\phi).
	\end{align*}
\end{proposition}
\begin{proof}
	This proposition will be proved in the similar way of Proposition \ref{as}.
\end{proof}
\begin{proposition}
	Let $\psi$ be an even invertible and $\chi$-symmetric element of the centroid of $\mathcal{B}_{\rho,\phi,\chi}$ such that $\psi\circ\phi=\phi\circ\psi$.
	Then $\mathcal{B}$ together with two structures $[.,.,.]^{\psi}, [.,.,.]_{\psi}^{\psi}$ and the even bilinear form $\chi_{\psi}$ defined by $\chi_{\psi}(f,g) = \chi(\psi(f), g)$ makes two quadratic 3-Hom-$\rho$-Lie algebras.
\end{proposition}
\begin{proof}
	According to the proof of Proposition \ref{gh}, it is enough to check that $\chi_{\psi}(\phi(f),g)=\chi_{\psi}(f,\phi(g))$. For this, we have
	$$\chi_{\psi}(\phi(f),g)=\chi(\psi\circ\phi(f),g)=\chi(\phi\circ\psi(f),g)=\chi(\psi(f),\phi(g))=\chi_{\psi}(f,\phi(g)).$$
\end{proof}
\begin{definition}
	Let $g$ be a $G$-graded vector space, $\mu : g\times g\times g\longrightarrow g$ be an even 3-linear map (i.e. $\mu(g_a,g_b,g_c)\subset g_{a+b+c}$) and $\alpha:g\longrightarrow g$ be an even homomorphism such that for elements $a,b,c,d,e\in g$
	$$\mu(\alpha(a),\alpha(b), \mu(c,d,e)) = \mu(\mu(a,b,c), \alpha(d),\alpha(e)),$$
	then $(g, \mu, \alpha)$ is called a 3-Hom-associative algebra.
	If the bilinear map $\mu$ is the $\rho$-symmetry with respect to the displacement of every two elements, the 3-Hom-associative algebra $(g, \mu,\alpha)$ is called commutative. We denote this algebra by $g_{\rho,\phi,C}$.
\end{definition}
\begin{definition}
 A quadratic structure on 3-Hom-associative algebra $g$ is a $\rho$-symmetric, invariant and non-degenerate bilinear form $\chi_g$ on $g$ such that $\alpha$ is $\chi_g$-symmetric. Let us denoted by $g_{\rho,\chi}$ a quadratic 3-Hom-associative algebra.
\end{definition}
\begin{theorem}
Consider $g_{\rho,\phi,\chi,C}$ and  $\mathcal{B}_{\rho,\phi,\chi}$. Then the tensor product $E=\mathcal{B}\otimes g$ equipped with the structures
\begin{align*}
[f\otimes a, g\otimes b, h\otimes c]_E &= \rho(a,g+h)\rho(b,h)[f,g,h]_\mathcal{B}\otimes\mu(a, b,c),\\
\chi_E(f\otimes a, g\otimes b) &=\rho(a,g)\chi_A(f,g)\chi_g(a,b),\\
\phi_E(f\otimes a, g\otimes b)&=\phi(f)\otimes\alpha(a),\quad \forall f,g,h\in Hg(\mathcal{B}), a, b,c \in Hg(g),
\end{align*}
is a quadratic 3-Hom-$\rho$-Lie algebra.
\end{theorem}
\begin{proof}
By the proof of Theorem \ref{a1}, we can easily show that $(E, [.,.,.]_E, \rho,\phi_E)$
is a 3-Hom-$\rho$-Lie algebra and $\chi_E$ is $\rho$-symmetric and invariant. It is enough to check that $\phi_E$ is $\chi_E$-symmetric, which completes the proof. For this, we have
\begin{align*}
\chi_E(\phi_E(f\otimes a), g\otimes b)&=\chi_E(\phi(f)\otimes\alpha(a), g\otimes b)=\rho(a,g)\chi(\phi(f),g)\chi_g(\alpha(a), b)\\
&=\rho(a,g)\chi(f,\phi(g))\chi_g(a, \alpha(b))=\chi_E(f\otimes a,\phi_E(g\otimes b)).
\end{align*}
\end{proof}
\textbf{Ideals of 3-Hom-$\rho$-Lie algebras:}
In this part, we introduce the definitions of subalgebra and ideal of $\mathcal{B}_{\rho,\phi}$ and give some properties related to ideals of $\mathcal{B}_{\rho,\phi}$.
\begin{definition}
	A Hom-subalgebra of $\mathcal{B}_{\rho,\phi}$ is defined as a sub-vector space $I\subseteq \mathcal{B}_{\rho,\phi}$ with the property $[I,I,I]_\mathcal{B}\subseteq I$ and $\phi(I)\subseteq I$. $I$ also is called a Hom ideal of $\mathcal{B}$ if $\phi(I)\subseteq I$ and $[I,\mathcal{B}_{\rho,\phi},\mathcal{B}_{\rho,\phi}]_\mathcal{B}\subseteq I$.
\end{definition}
\begin{lemma}
	$Z(\mathcal{B}_{\rho,\phi})$, the center of multiplicative quadratic 3-Hom-$\rho$-Lie algebra $\mathcal{B}_{\rho,\phi,\chi}$, is an Hom-ideal of $\mathcal{B}_{\rho,\phi}$.
\end{lemma}
\begin{proof}
	It is easy to see that $[Z(\mathcal{B}_{\rho,\phi}),\mathcal{B}_{\rho,\phi},\mathcal{B}_{\rho,\phi}]\subseteq \mathcal{B}_{\rho,\phi}$. We assume that $f\in Z(\mathcal{B}_{\rho,\phi})$ and $g,h,m\in \mathcal{B}$, so
	\begin{align*}
	\chi([\phi(f), g, h]_\mathcal{B}, m)&=\chi(\phi(f), [g,h,m]_\mathcal{B})=\chi(f,\phi[g,h,m]_\mathcal{B})\\
	&=\chi(f,[\phi(g),\phi(h), \phi(m)]_\mathcal{B})
	=\chi([f,\phi(g),\phi(h)]_\mathcal{B}, \phi(m))=0.
	\end{align*}
	The last equality holds, since $f\in Z(\mathcal{B}_{\rho,\phi})$. On the other hand, since $\chi$ is non-degenerate, thus the above relation gives $[\phi(f), g, h]_\mathcal{B}$ for any $g,h\in \mathcal{B}$. This implies that $\phi(f)\in Z(\mathcal{B}_{\rho,\phi})$ and therefore $Z(\mathcal{B}_{\rho,\phi})$ is an Hom-ideal of $\mathcal{B}_{\rho,\phi}$.
\end{proof}
\begin{lemma}
	The orthogonal $I^{\perp}$ of $I$ ($I$ is a Hom-ideal of  $\mathcal{B}_{\rho,\phi,\chi}$) with respect to $\chi$ is a Hom-ideal of $\mathcal{B}_{\rho,\phi}$.
\end{lemma}
\begin{proof}
	It is clear that $[I^{\perp}, \mathcal{B}_{\rho,\phi}, \mathcal{B}_{\rho,\phi}]\subseteq I^{\perp}$. Also, for $f\in I^{\perp}$, $g\in I$, we have $B(\phi(f), g)=B(f,\phi(g))=0$, since $\phi(I)\subseteq I$ and $f\in I^{\perp}$.
\end{proof}
\textbf{Symplectic structure of 3-Hom-$\rho$-Lie algebras:}
This part is devoted to study of symplectic structure, metrics and their properties. Also, we show that a symplectic structure $\omega$ may be defined on $\mathcal{B}_{\rho,\phi}$ if and only if there exists an invertible derivation of $\mathcal{B}_{\rho,\phi}$ that is antisymmetric with respect to metric.
\begin{definition}\label{a20}
	If $\omega\in\wedge^2\mathcal{B}^*$ is a non-degenerate and $\rho$-skew-symmetric bilinear form
	such that 
	\begin{align*}
	&\omega([f_1, f_2, f_3]_\mathcal{B}, \phi(f_4))-\rho(f_1, f_2+f_3+f_4)\omega([f_2, f_3, f_4]_\mathcal{B}, \phi(f_1))\\
	&\ \ \ +\rho(f_1+f_2,f_3+f_4)\omega([f_3, f_4, f_1]_\mathcal{B}, \phi(f_2))-\rho(f_1+f_2+f_3, f_4)\omega([f_4, f_1, f_2]_\mathcal{B},\phi(f_3))=0,
	\end{align*}
	then $\omega$ is called a symplectic structure and $(\mathcal{B}, [.,.,.]_\mathcal{B},\rho,\phi,\omega)$ is said to be a symplectic 3-Hom-$\rho$-Lie algebra.
\end{definition}
We will say that $(\mathcal{B},[.,.,.]_\mathcal{B},\rho,\phi, \chi, \omega)$ is a quadratic symplectic 3-Hom-$\rho$-Lie algebra if $(\mathcal{B},\chi)$ is quadratic 
and $(\mathcal{B},\omega)$ is symplectic.
\begin{example}
	Let $(\mathcal{B}, [.,.,.]_\mathcal{B},\rho,\omega)$ be the $4$-dimensional symplectic 3-$\rho$-Lie algebra provided in Example \ref{Ex.123}. We consider the multiplex $(\mathcal{B},[.,.,.]_{\phi},\rho,\phi)$ defining a 3-Hom-$\rho$-Lie algebra, according to Theorem \ref{Th.124}, such that the bracket $[.,.,.]_{\phi}$ and the even linear map $\phi$ are defined by
	\begin{align*}
	&[l_1,l_2,l_3]_{\phi}=l_4,\quad \phi(l_4)=l_4,\\
	&[l_1,l_2,l_4]_{\phi}=l_3,\quad \phi(l_3)=-l_3,\\
	&[l_1,l_3,l_4]_{\phi}=l_2,\quad \phi(l_2)=-l_2,\\
	&[l_2,l_3,l_4]_{\phi}=l_1,\quad \phi(l_1)=-l_1.
	\end{align*}
	Moreover, $(\mathcal{B},[.,.,.]_{\phi},\rho,\phi)$ together with the bilinear map $\omega$ is given in Example \ref{Ex.123} is a $4$-dimensional symplectic 3-Hom-$\rho$-Lie algebra.
\end{example}
In the following, we define a derivation of $\mathcal{B}_{\rho,\phi}$ and examine the relationship between symplectic structures and $\varphi$-symmmetric invertible derivations.
\begin{definition}
A linear map $D$ on $\mathcal{B}_{\rho,\phi}$ satisfying  
\begin{align*}
&D\circ \phi=\phi\circ D,\\
&D[f,g,h]_\mathcal{B}= [D(f), g, h] +\rho(D, f) [f, D(g), h] + \rho(D, f+g)[f,g,D(h)],
\end{align*}
is called a derivation of $\mathcal{B}_{\rho,\phi}$. We denote by $Der(\mathcal{B}_{\rho,\phi})$ the space of all derivations of $\mathcal{B}_{\rho,\phi}$.
\end{definition}
\begin{definition}
$(\mathcal{B}, [.,.,.]_\mathcal{B},\rho,\phi,\varphi,\mathbb{K})$ is called a metric 3-Hom-$\rho$-Lie algebra, if
\begin{enumerate}
\item[i.]
$\mathbb{K}$ is a field $(\mathbb{C}$ or $\mathbb{R})$,
\item[ii.]
$(\mathcal{B}, [.,.,.]_\mathcal{B},\rho,\phi)$ is a 3-Hom-$\rho$-Lie algebra,
\item[iii.]
 $\varphi:\mathcal{B}\times \mathcal{B}\longrightarrow \mathbb{K}$ is a
non-degenerate $\rho$-symmetric bilinear form on $\mathcal{B}$ satisfying the following condition
\begin{align*}
&\varphi([f,g,h]_{\mathcal{B}},\phi(z)) +\rho(f+g,h) \varphi(\phi(h), [f,g,z]_{\mathcal{B}}) = 0,\quad\forall f,g,h,z\in Hg(\mathcal{B}).
\end{align*}
\end{enumerate} 
Let us denote by ${\mathcal{B}}_{\rho,\phi,\varphi}$ the metric 3-Hom-$\rho$-Lie algebra $(\mathcal{B}, [.,.,.]_\mathcal{B},\rho,\phi,\varphi,\mathbb{K})$.
\end{definition}
\begin{proposition}
 There exists a symplectic structure on $\mathcal{B}_{\rho,\phi}$ if and only if there exists an invertible derivation $D\in {\rm Der_{\varphi}(\mathcal{B}_{\rho,\phi})}$.
\end{proposition}
\begin{proof}
	It is enough to define $\omega(f,g)=\varphi(D(f),g)$.
\end{proof}
 For any symmetric bilinear form $\chi^{\prime}$ on $\mathcal{B}_{\rho,\phi,\chi}$, there is an associated map $D:\mathcal{B}\longrightarrow \mathcal{B}$ satisfying
\begin{equation}\label{qq}
\chi^{\prime}(f,g) = \chi(D(f),g),\quad \forall f,g\in Hg(\mathcal{B}).
\end{equation}
Since $\chi$ and $\chi^{\prime}$ are symmetric, then $D$ is symmetric with respect to $\chi$, i.e., $$\chi(D(f),g) =\rho(D,f)\chi(f,D(g)),\quad \forall f,g\in Hg(\mathcal{B}).$$
\begin{lemma}
	Let $(\mathcal{B},[.,.,.]_\mathcal{B},\rho,\phi,\chi)$ be a quadratic 3-Hom-$\rho$-Lie algebra and $\chi^{\prime}$ defined by \eqref{qq}. Then 
		$\chi^{\prime}$ is invariant if and only if
		\begin{equation}
		D([f,g,h]_\mathcal{B}) = [D(f),g,h]_\mathcal{B} = \rho(D,f)[f,D(g),h]_\mathcal{B}=\rho(D,f+g)[f,g,D(h)]_\mathcal{B}.
		\end{equation}
Also, $\chi^{\prime}$ is non-degenerate if and only if $D$ is invertible and $\phi$ is $\chi^{\prime}$-symmetric if and only if $D\circ\phi =\phi\circ D$.
\end{lemma}
\begin{proof}
By the proof of Lemma \ref{lem 25}, it is enough to show that $\chi^{\prime}(\phi(f),g)=\chi^{\prime}(f,\phi(g))$. For this, we have
$$\chi^{\prime}(\phi(f),g)=\chi(D\circ\phi(f),g)=\chi(\phi\circ D(f),g)=\chi(D(f),\phi(g))=\chi^{\prime}(f,\phi(g)).$$
\end{proof}
\begin{definition}
	A $\rho$-symmetric map $D:\mathcal{B}_{\rho,\phi}\longrightarrow \mathcal{B}_{\rho,\phi}$ satisfying \eqref{eq.12} is called a centromorphism of $\mathcal{B}_{\{\},\phi}$. We denote by $\mathscr{C}(\mathcal{B}_{\rho,\phi})$ the space of all centromorphisms of $\mathcal{B}_{\rho,\phi}$.
\end{definition}
\begin{proposition}
	Let $\delta\in {\rm Der_\chi(\mathcal{B}_{\rho,\phi})}$ and $D\in\mathscr{C}(\mathcal{B}_{\rho,\phi})$ such that $D\circ \delta=\delta\circ D$ and $D\circ \phi=\phi\circ D$ . Then $D\circ \delta$ is also a $\chi$-antisymmetric derivation of $\mathcal{B}_{\rho,\phi}$.
\end{proposition}
\begin{proof}
By the proof of Proposition \ref{pro 26}, it is enough to show that $(D\circ \delta)\circ\phi =\phi\circ (D\circ \delta)$. Since $\delta\circ\phi =\phi\circ\delta$, we have
$$(D\circ \delta)\circ\phi=D\circ (\delta\circ\phi)=D\circ(\phi\circ\delta )=(D\circ\phi)\circ\delta=(\phi\circ D)\circ\delta=\phi\circ (D\circ\delta).$$
\end{proof}
%-----------------------------------------------------------------------------------------------
%--------------------------------------------------------------------------------------------------------------------------------------------------------------------------
\subsection{Representation of 3-Hom-$\rho$-Lie algebras}
The representation theory of 3-Hom-$\rho$-Lie algebras is studied in \cite{PZCG}. In this part, we recall main definition of representation and some properties, then we add a brief details which are not discussed in that article.
 
 Consider the multiplicative 3-Hom-$\rho$-Lie algebra $\mathcal{B}_{\rho,\phi}$. The fundamental set $\mathcal{L}=\wedge^2\mathcal{B}_{\rho,\phi}$ together with operation
\begin{equation}
[(f_1,f_2), (g_1,g_2)]_{\mathcal{L}}=([f_1, f_2, g_1],\phi(g_2))+\rho(f_1+f_2,g_1)(\phi(g_1),[f_1,f_2,g_2]),
\end{equation}
and the even linear map $\phi_1:\mathcal{L}\longrightarrow \mathcal{L}$ defined by $\phi_1(f_1,f_2)=(\phi(f_1),\phi(f_2))$ is the multiplicative Hom-$\rho$-Lie algebra.
\begin{definition}\cite{PZCG}\label{461}
	A bilinear map $\mu:\mathcal{L}\longrightarrow {\rm gl(V)}$ is called a representation of  $\mathcal{B}_{\rho,\phi}$ on vector space $V$ with respect to $\beta\in {\rm gl(V)}$ if the following equalities are satisfied	
	\begin{align}
	\mu[(f_1,f_2),(g_1,g_2)]_{_\mathcal{L}}\circ \beta&=\mu(\phi_1(f_1,f_2))\mu(g_1,g_2)-\rho(f_1+f_2,g_1+g_2)\mu(\phi_1(g_1,g_2))\mu(f_1,f_2),\label{1}\\
	\mu([g_1,g_2,g_3],\phi(f))\circ \beta&=\mu(\phi_1(g_1,g_2))\mu(g_3,f)+\rho(g_1,g_2+g_3)\mu(\phi_1(g_2,g_3))\mu(g_1,f)\label{2}\\
	&\ \ \  +\rho(g_1+g_2,g_3)\mu(\phi_1(g_3,g_1))\mu(g_2,f),\nonumber\\
	\mu(\phi(g),[f_1,f_2,f_3])\circ \beta&=\rho(g,f_1+f_2)\mu(\phi_1(f_1,f_2))\mu(g,f_3)+\rho(g,f_2+f_3)\rho(f_1,f_2+f_3)\mu(\phi_1(f_2,f_3))\mu(g,f_1)\\
	&\ \ \  +\rho(g,f_1+f_3)\rho(f_1+f_2,f_3)\mu(\phi_1(f_3,f_1))\mu(g,f_2).\nonumber
	\end{align}
\end{definition}
If $\mathcal{B}_{\rho,\phi}$ is multiplicative, in addition to the above conditions, the following condition must be hold too 
$$\mu(\phi(f),\phi(g))\circ\beta=\beta\circ\mu(f,g).$$
\begin{lemma}\label{sa5}
	Let $(V,\mu,\beta)$ be a representation of $\mathcal{B}_{\rho,\phi}$. Then, we have
	\begin{align*}
	0&=\rho(f_1+f_2,g_1)\mu(\phi(g_1), [f_1,f_2,g_2]_\mathcal{B})\circ \beta +\mu([f_1,f_2,g_1]_\mathcal{B}, \phi(g_2))\circ \beta\\
	&\quad +\rho(f_1+f_2,g_1+g_2)\rho(g_1+g_2,f_1)\mu(\phi(f_1), [g_1, g_2, f_2]_\mathcal{B})\circ \beta\\
	&\ \ \ +\rho(f_1+f_2, g_1+g_2)\mu([g_1, g_2, f_1]_\mathcal{B}, \phi(f_2))\circ \beta.
	\end{align*}
\end{lemma}
\begin{proof}
	For the proof refer to Lemma \ref{aa5}.
\end{proof}
If $V=\mathcal{B}_{\rho,\phi}$ and $\phi=\beta\in {\rm gl(\mathcal{B}_{\rho,\phi})}$, then the bilinear map $ad:\mathcal{B}\times \mathcal{B}\longrightarrow {\rm gl(\mathcal{B})}$ defined by $ad(f_1,f_2)(f_3)=[f_1,f_2,f_3]_{\mathcal{B}}$
is a representation of $\mathcal{B}_{\rho,\phi}$ with respect to $\beta=\phi$.
\begin{proposition}\label{pro 3.14}
Consider the 3-Hom-$\rho$-Lie algebra $\mathcal{B}_{\rho,\phi}$. Let $V$ be a graded vector space, $\beta\in {\rm gl(V)}$ and $\mu:\wedge^2\mathcal{B}\longrightarrow gl(V)$ be a bilinear map. Then $(V,\mu,\beta)$ is a representation of $\mathcal{B}_{\rho,\phi}$ if and only if  $\mathcal{B}\oplus V$ is a 3-Hom-$\rho$-Lie algebra with the following structures 
	\begin{align*}
	[f_1+v_1, f_2+v_2, f_3+v_3]^{\mu}_{\mathcal{B}\oplus V} &= [f_1, f_2, f_3]_\mathcal{B} +\mu(f_1, f_2)v_3\\
	&\ \ \ \ +\rho(f_1, f_2+f_3)\mu(f_2,f_3)v_1+\rho(f_1+f_2,f_3)\mu(f_3,f_1)v_2,\\
	\psi(f_1+v_1)&=\phi(f_1)+\beta(v_1),
	\end{align*}
	where $f_1, f_2, f_3\in Hg(\mathcal{B})$ and $v_1, v_2, v_3\in Hg(V)$.
\end{proposition}
Consider the representation $(V,\mu,\beta)$ of $\mathcal{B}_{\rho,\phi}$ and $V^{\star}$ as the dual of vector space $V$. We define bilinear map $\widetilde{\mu}:\mathcal{B}\times \mathcal{B}\longrightarrow End(V^{\star})$ by $\widetilde{\mu}(f_1,f_2)(\varrho)=-\rho(f_1+f_2,\varrho)\varrho\circ\mu(f_1,f_2)$, where $f_1,f_2\in Hg(\mathcal{B}),~~\varrho\in V^{\star}$ and set $\widetilde{\beta}(\varrho)=\varrho\circ\beta$. \\
Note that by Definition \ref{dd} and Lemma \ref{sa5}, we can deduce that if $(V,\mu,\beta)$ is a $\rho$-skew symmetric representation, then $(V^{\star},\widetilde{\mu},\widetilde{\beta})$ is a representation of $\mathcal{B}_{\rho,\phi}$.\\
By abuse the above notation, $(\mathcal{B}^*,ad^*)$ with respect to $\tilde{\beta}(\varrho)=\varrho\circ\phi$ is a representation of $\mathcal{B}_{\rho,\phi}$. So by Proposition \ref{pro 3.14}, $\mathcal{B}\oplus \mathcal{B}^*$ is a 3-Hom-$\rho$-Lie algebra together with the following structure
\begin{align*}
[f_1+\alpha_1, f_2+\alpha_2, f_3+\alpha_3]^{ad^*}_{\mathcal{B}\oplus \mathcal{B}^*} &= [f_1, f_2, f_3]_\mathcal{B} +ad^*(f_1, f_2)\alpha_3\\
&\ \ \ \ +\rho(f_1, f_2+f_3)ad^*(f_2,f_3)\alpha_1+\rho(f_1+f_2,f_3)ad^*(f_3,f_1)\alpha_2,\\
(\phi+\phi^*)(f+\alpha)&=\phi(f)+\alpha\circ\phi.
\end{align*}
Now, consider the bilinear map $\varphi(f+\alpha,g+\beta)=\alpha(g)+\rho(f,g)\beta(f)$. Then $(\mathcal{B}\oplus \mathcal{B}^*, [.,.,.]^{ad^*}_{\mathcal{B}\oplus \mathcal{B}^*}, \rho,\phi+\phi^*,\varphi)$ is a metric 3-Hom-$\rho$-Lie algebra.
\section{3-pre-Hom-$\rho$-Lie algebras}
This part is devoted to the 3-pre-Hom-$\rho$-Lie algebras, which are studied similar to the classical case in Section 4. Actually, we give the similar results to Section 4 in the case of 3-pre-Hom-$\rho$-Lie algebras.

\begin{definition}
$(\mathcal{B},\{.,.,.\},\rho,\phi)$ is called a 3-pre-Hom-$\rho$-Lie algebra if
\item[i.]
$\mathcal{B}$ is a $G$-graded vector space,
\item[ii.]
$\phi:\mathcal{B}\longrightarrow \mathcal{B}$ is an even linear map,
\item[iii.]
$\{.,.,.\}:\otimes^3 \mathcal{B}\longrightarrow \mathcal{B}$ is a trilinear map satisfying the following relations
	\begin{align*}
	|\{f_1, f_2, f_3\}|&=|f_1| +|f_2| +|f_3|,\\
	\{f_1, f_2, f_3\}&=-\rho(f_1,f_2)\{f_2, f_1, f_3\},\\
	\{\phi(f_1),\phi(f_2), \{g_1, g_2,g_3\}\}&=\{[f_1,f_2,g_1]_c, \phi(g_2), \phi(g_3)\} +\rho(f_1+f_2,g_1)\{\phi(g_1), [f_1,f_2,g_2]_c, \phi(g_3)\}\\
	&\quad +\rho(f_1+f_2,g_1+g_2)\{\phi(g_1), \phi(g_2), \{f_1,f_2,g_3\}\},\\
	\{[f_1,f_2,f_3]_c, \phi(g_1), \phi(g_2)\}&=\{\phi(f_1), \phi(f_2), \{f_3, g_1, g_2\}\} +\rho(f_1,f_2+f_3)\{\phi(f_2), \phi(f_3), \{f_1, g_1,g_2\}\}\\
	&\quad +\rho(f_1+f_2,f_3)\{\phi(f_3), \phi(f_1), \{f_2,g_1,g_2\}\},
	\end{align*}
	where 
	\begin{align*}
	[f_1, f_2, f_3]_c=\{f_1, f_2, f_3\}+\rho(f_1, f_2+f_3)\{f_2, f_3, f_1\}+\rho(f_1+f_2,f_3)\{f_3, f_1, f_2\}.
	\end{align*}
We denoted by $\mathcal{B}_{\{\},\phi}$ the 3-pre-Hom-$\rho$-Lie algebra $(\mathcal{B},\{.,.,.\},\rho,\phi)$.
\end{definition}
For a 3-pre-Hom-$\rho$-Lie algebra $(\mathcal{B}, \{.,.,.\}, \rho, \phi)$, it is easy to see that $(\mathcal{B}, [.,.,.]_c,\rho,\phi)$ is a 3-Hom-$\rho$-Lie algebra. We denote it by $\mathcal{B}^c_{\rho,\phi}$.\\

Let us define two multiplications  $L:\wedge^2\mathcal{B}\longrightarrow gl(\mathcal{B})$ and $R:\otimes^2\mathcal{B}\longrightarrow gl(\mathcal{B})$ on $\mathcal{B}_{\{\},\phi}$, which is called respectively the left and right multiplications, by  $L(f,g)h=\{f,g,h\}$ and $R(f,g)h=\rho(f+g,h)\{h,f,g\}$. Then $(\mathcal{B},L)$ is a representation of $\mathcal{B}^c_{\rho,\phi}$ with respect to $\beta=\phi$.
\begin{definition}
The triple $(V,\mu,\tilde{\mu})$ consisting of a $G$-graded vector space $V$, a representation $\mu$ of $\mathcal{B}^c_{\rho,\phi}$ and a bilinear map $\tilde{\mu}:\otimes^2\mathcal{B}\longrightarrow gl(V)$
	is called a representation of $\mathcal{B}_{\{\},\phi}$ with respect to $\beta\in{\rm gl(V)}$ if for all homogeneous elements $f_1, f_2, f_3, f_4\in Hg(\mathcal{B})$, the following equalities hold
	\begin{align*}
	\tilde{\mu}(\phi(f_1), \{f_2, f_3, f_4\})\circ \beta&=\rho(f_1, f_2+f_3)\mu(\phi(f_2), \phi(f_3))\tilde{\mu}(f_1, f_4)\\
	&\ \ \ +\rho(f_1, f_3+f_4)\rho(f_2, f_3+f_4)\tilde{\mu}(\phi(f_3), \phi(f_4))\tilde{\mu}(f_1, f_2)\\
	&\ \ \ -\rho(f_1, f_2+f_4)\rho(f_3, f_4)\tilde{\mu}(\phi(f_2), \phi(f_4))\tilde{\mu}(f_1, f_3)\\
	&\ \ \ +\rho(f_1+f_2,f_3+f_4)\tilde{\mu}(\phi(f_3), \phi(f_4))\mu(f_1, f_2)\\
	&\ \ \ +\rho(f_1, f_2+f_3+f_4)\rho(f_3, f_4)\tilde{\mu}(\phi(f_2), \phi(f_4))\tilde{\mu}(f_3, f_1)\\
	&\ \ \ -\rho(f_1, f_2+f_4)\rho(f_3, f_4)\tilde{\mu}(\phi(f_2), \phi(f_4))\mu(f_1, f_3)\\
	&\ \ \ -\rho(f_2, f_3+f_4)\rho(f_1, f_2+f_3+f_4)\tilde{\mu}(\phi(f_3), \phi(f_4))\tilde{\mu}(f_2,f_1),\\
	&\\
	\mu(\phi(f_1), \phi(f_2))\tilde{\mu}(f_3, f_4)&=\rho(f_1+f_2, f_3+f_4)\tilde{\mu}(\phi(f_3), \phi(f_4))\mu(f_1, f_2)\\
	&\ \ \ -\rho(f_1+f_2, f_3+f_4)\rho(f_1, f_2)\tilde{\mu}(\phi(f_3), \phi(f_4))\tilde{\mu}(f_2, f_1)\\
	&\ \ \ +\rho(f_1+f_2, f_3+f_4)\tilde{\mu}(\phi(f_3), \phi(f_4))\tilde{\mu}(f_1, f_2)+\tilde{\mu}([f_1, f_2, f_3]_\mathcal{B}, \phi(f_4))\circ \beta\\
	&\ \ \ +\rho(f_1+f_2, f_3)\tilde{\mu}(\phi(f_3), \{f_1, f_2, f_4\}),\\
	&\\
	\tilde{\mu}([f_1, f_2, f_3]_\mathcal{B}, \phi(f_4))\circ \beta&=\mu(\phi(f_1), \phi(f_2))\tilde{\mu}(f_3, f_4)+\rho(f_1, f_2+f_3)\mu(\phi(f_2), \phi(f_3))\tilde{\mu}(f_1, f_4)\\
	&\ \ \ +\rho(f_1+f_2,f_3)\mu(\phi(f_3), \phi(f_1))\tilde{\mu}(f_2, f_4),\\
	&\\
	\tilde{\mu}(\phi(f_3), \phi(f_4))\mu(f_1, f_2)&=\rho(f_1, f_2)\mu(\phi(f_3), \phi(f_4))\tilde{\mu}(f_2, f_1)-\tilde{\mu}(\phi(f_3), \phi(f_4))\tilde{\mu}(f_1, f_2)\\
	&\ \ \ +\rho(f_3+f_4, f_1+f_2)\mu(\phi(f_1), \phi(f_2))\tilde{\mu}(f_3, f_4)\\
	&\ \ \ -\rho(f_3+f_4, f_1+f_2)\rho(f_1, f_2)\tilde{\mu}(\phi(f_2), \{f_1, f_3, f_4\})\circ \beta\\
	&\ \ \ +\rho(f_3+f_4, f_1+f_2)\tilde{\mu}(\phi(f_1), \{f_2, f_3, f_4\})\circ \beta.
	\end{align*}
\end{definition}
For example $(\mathcal{B},L,R,\phi)$ is a representation of $\mathcal{B}_{\{\},\phi}$.
Let us to define the operation $\{.,.,.\}_{\mu, \tilde{\mu}}:\otimes^2(\mathcal{B}\oplus V)\longrightarrow \mathcal{B}\oplus V$ and the linear map $\psi: \mathcal{B}\oplus V\longrightarrow \mathcal{B}\oplus V$ by
	\begin{align*}
	\{f_1+v_1, f_2+v_2, f_3+v_3\}_{\mu, \tilde{\mu}}&=\{f_1, f_2, f_3\}+\mu(f_1, f_2)v_3+\rho(f_1, f_2+f_3)\tilde{\mu}(f_2, f_3)v_1-\rho(f_2, f_3)\tilde{\mu}(f_1, f_3)v_2,\\
	\psi(f_1+v_1)&=\phi(f_1)+\beta(v_1),
	\end{align*}
for all $f_1, f_2, f_3\in Hg(\mathcal{B})$ and $v_1, v_2, v_3\in V$. Then $(\mathcal{B}\oplus V, \{.,.,.\}_{\mu, \tilde{\mu}}, \rho, \psi)$ is a 3-pre-Hom-$\rho$-Lie algebra.\\

Consider the operator $\zeta:\otimes^2V\longrightarrow\otimes^2V$ given by $ \zeta(f_1\otimes f_2)=f_2\otimes f_1$, where $V$ is a graded vector space and $f_1\otimes f_2\in\otimes^2V$.\\

Using the above notation, in the following proposition, we give a representation of $\mathcal{B}^c_{\rho,\phi}$:
\begin{proposition}
Let $(V, \mu, \tilde{\mu},\beta)$ be a representation of $\mathcal{B}_{\{\},\phi}$. Then there exists a representation $\nu:\wedge^2\mathcal{B}\longrightarrow gl(V)$ of $\mathcal{B}^c_{\rho,\phi}$ on $V$ that is given by $\nu(f_1, f_2)=(\mu-\rho(f_1, f_2)\tilde{\mu}\xi+\tilde{\mu})(f_1, f_2)$ for all $f_1, f_2\in Hg(\mathcal{B})$.
\end{proposition}
The above proposition conclude that if $(V,\mu, \tilde{\mu})$ is a representation of $\mathcal{B}_{\{\},\phi}$, then the 3-pre-Hom-$\rho$-Lie algebras $(\mathcal{B}\oplus V,\{.,.,.\}_{\mu, \tilde{\mu}},\rho,\psi)$ and $(\mathcal{B}\oplus V,\{.,.,.\}_{\nu,0},\rho,\psi)$ have the same sub-adjacent 3-Hom-$\rho$-Lie algebra given by \eqref{a12}.

Let us define the dual of the representation $(V,\nu)$ by $\nu^*(f_1, f_2)=\mu^*-\rho(f_1,f_2)\tilde{\mu}^*\xi+\tilde{\mu}^*(f_1,f_2)$. In the following proposition, we will talk about the dual of the representation $(V,\mu,\tilde{\mu})$:
\begin{proposition}
Consider $\mathcal{B}_{\{\},\phi}$ equipped with the $\rho$-skew-symmetric representation $(V,\mu,\tilde{\mu},\beta)$. Then $(V^*,\nu^*, -\tilde{\mu}^*,\beta)$ is a representation of $\mathcal{B}_{\{\},\phi}$, which is called the dual representation of $(V,\mu, \tilde{\mu})$.
\end{proposition}
\begin{remark}
Consider $\mathcal{B}_{\{\},\phi}$ as a 3-pre-Hom-$\rho$-Lie algebra with the $\rho$-skew-symmetric representation $(V,\mu, \tilde{\mu},\beta)$. Then the 3-pre-Hom-$\rho$-Lie algebras $(\mathcal{B}\oplus V^*, \{.,.,.\}_{\mu^*,0},\rho,\psi)$ and $(\mathcal{B}\oplus V^*,\{.,.,.\}_{\nu^*, -\tilde{\mu}^*},\rho,\psi)$ have the same sub-adjacent 3-Hom-$\rho$-Lie algebra $(\mathcal{B}^c\oplus V^*,[.,.,.]^{\mu^*}_c,\rho,\phi)$ with the representation $(V^*,\mu^*,\tilde{\beta})$.
\end{remark}
In the following definition, we define the $\rho$-Hom-$\mathcal{O}$-operator associated to representation $(V,\mu, \beta)$. The importance of defining this operator is that it can be used to construct a 3-pre-Hom-$\rho$-Lie algebra structure on representation $V$.
\begin{definition}\label{z12}
Let $(V,\mu, \beta)$ be a representation of $\mathcal{B}_{\rho,\phi}$. If an even linear operator $T : V \longrightarrow \mathcal{B}_{\rho,\phi}$  satisfies
\begin{align*}
T\circ \beta&=\phi\circ T,\\
[Tx, Ty, Tz]_\mathcal{B} &= T (\mu(Tx, Ty)z +\rho(x, y+z)(\mu(Ty, Tz)x + \rho(x+y, z)\mu(Tz, Tx)y),\quad \forall x,y,z \in V, 
\end{align*}
it is called a $\rho$-Hom-$\mathcal{O}$-operator associated to representation $(V,\mu, \beta)$.
\end{definition}
\begin{lemma}
Let $(V, \mu, \beta)$ be a $\rho$-skew-symmetric representation of $\mathcal{B}_{\rho,\phi}$ and $T : V \longrightarrow \mathcal{B}_{\rho,\phi}$ is a $\rho$-Hom-$\mathcal{O}$-operator associated to $(V, \mu, \beta)$. Then $(V, \{.,.,.\}, \rho, \psi)$ is a 3-pre-Hom-$\rho$-Lie algebra, where
\begin{align}
\psi=\beta,\\
\{x,y,z\} &= \mu(Tx, Ty)z,\quad \forall x,y,z\in V.\label{460}
\end{align}
\end{lemma}
\begin{proof}
	Suppose that $x,y,z\in Hg(V)$. It is easy to see that 
	\begin{align*}
	\{x,y,z\}=-\rho(x,y)\{y,x,z\},\qquad T[x,y,z]_c=[Tx,Ty,Tz].
	\end{align*}
	Also, for $x_1, x_2, x_3, x_4, x_5\in Hg(V)$, we have
	\begin{align}
	\{\psi(x_1), \psi(x_2), \{x_3, x_4, x_5\}\}&=\mu(T\circ \psi (x_1), T\circ \psi (x_2))\mu(Tx_3, Tx_4)x_5\nonumber\\
	&=\mu(\phi\circ T (x_1), \phi\circ T (x_2))\mu(Tx_3, Tx_4)x_5,\label{s1}\\
	\{[x_1, x_2, x_3]_c,\psi(x_4), \psi(x_5)\}&=\mu([Tx_1, Tx_2,Tx_3],T\circ \psi(x_4))x_5\nonumber\\
	&=\mu([Tx_1, Tx_2,Tx_3],\phi\circ T(x_4))x_5,\label{s2}\\
	\rho(x_1+x_2, x_3)\{\psi(x_3), [x_1, x_2, x_4]_c, \psi(x_5)\}&=\rho(x_1+x_2, x_3)\mu(T\circ\psi (x_3), [Tx_1, Tx_2, Tx_4])x_5\nonumber\\
	&=\rho(x_1+x_2, x_3)\mu(\phi\circ T (x_3), [Tx_1, Tx_2, Tx_4])x_5,\label{s3}\\
	\rho(x_1+x_2, x_3+x_4)\{\psi(x_3), \psi(x_4),\{x_1,x_2,x_5\}\}&=\rho(x_1+x_2, x_3+x_4)\mu(T\circ \psi(x_3), T\circ \psi(x_4))\mu(Tx_1, Tx_2)x_5\nonumber\\
	&=\rho(x_1+x_2, x_3+x_4)\mu(\phi\circ T(x_3), \phi\circ T(x_4))\mu(Tx_1, Tx_2)x_5.\label{s4}
	\end{align}
	By Definition \ref{461} and the equalities (\ref{s1})--(\ref{s4}), we get the result.
\end{proof}
\begin{lemma}
Consider $\mathcal{B}_{\rho,\phi}$ as a 3-Hom-$\rho$-Lie algebra. Then the following assertions are equivalent
	\begin{itemize}
		\item[i.]
		existence of a compatible 3-pre-Hom-$\rho$-Lie algebra structure,
		\item[ii.]
		existence of an invertible $\rho$-$\mathcal{O}$-operator.
	\end{itemize}
\end{lemma}
\begin{proof}
Refer to the proof of Lemma \ref{459}.
\end{proof}
\textbf{Phase Space:} In the following, we extend the study of phase space of a 3-$\rho$-Lie algebra to 3-Hom-$\rho$-Lie algebras  and show that a 3-Hom-$\rho$-Lie algebra has a phase space if and only if it is sub-adjacent to a 3-pre-Hom-$\rho$-Lie algebra. 
\begin{proposition}%\label{ass}
	Consider the symplectic 3-Hom-$\rho$-Lie algebra $(\mathcal{B}, [.,.,.]_\mathcal{B},\rho,\phi,\omega)$. Then the following structure defines a compatible 3-pre-Hom-$\rho$-Lie algebra structure on $\mathcal{B}$
	$$\omega(\{f,g,h\},\phi(s)) = -\rho(f+g,h)\omega(\phi(h), [f,g,s]_\mathcal{B}),~~\forall f,g,h,s\in Hg(\mathcal{B}).$$
\end{proposition}
\begin{proof}
Let us define a linear map $T:\mathcal{B}^*\longrightarrow \mathcal{B}$ by  $\omega(\phi (f), g)=\omega(f, \phi (g))=(T^{-1}(f))(g)$, or equivalently, $\omega(T\alpha, \phi(g)) =\alpha(g)$ for all $f,g\in Hg(\mathcal{B})$ and $\alpha\in\mathcal{B}^*$. Since $\omega$ is a symplectic structure, we deduce that $T$ is an invertible $\rho$-$\mathcal{O}$-operator associated to $(\mathcal{B}^*, Ad^*)$. Also, by \eqref{457} and \eqref{460}, there there exists a compatible 3-$\rho$-pre-Lie algebra on $B$ given by $\{f,g,h\} = T (ad^*(f,g)T^{-1}(h))$. So
\begin{align*}
\omega(\{f,g,h\},\phi(s))&=\omega(T (ad^*(f,g)T^{-1}(h)),\phi(s))=ad^*(f,g)T^{-1}(h)(s)\\
&=-\rho(f+g,h)T^{-1}(h)[f,g,s]_{\mathcal{B}}=-\rho(f+g,h)\omega(\phi(h),[f,g,s]_{\mathcal{B}}).
\end{align*}
\end{proof}
\begin{definition}
	$(\mathcal{B}\oplus \mathcal{B}^*, [.,.,.]_{\mathcal{B}\oplus \mathcal{B}^*},\rho,\phi+\phi^*, \omega)$ is called a phase space of the 3-Hom-$\rho$-Lie algebra $\mathcal{B}$, if
	\begin{enumerate}
		\item[i.]
		$(\mathcal{B}\oplus \mathcal{B}^*,[.,.,.]_{\mathcal{B}\oplus \mathcal{B}^*},\rho,\phi+\phi^*)$ is a 3-Hom-$\rho$-Lie algebra,
		\item[ii.]
		$\omega$ is a natural non-degenerate skew-symmetric bilinear form on $\mathcal{B}\oplus \mathcal{B}^*$ given by 
		\begin{equation}\label{112}
		\omega(f+\alpha, g+\beta)=\alpha(g)-\rho(f,g)\beta(f),
		\end{equation}
		such that $(\mathcal{B}\oplus \mathcal{B}^*, [.,.,.]_{\mathcal{B}\oplus \mathcal{B}^*},\rho,\phi+\phi^*, \omega)$ is a symplectic 3-Hom-$\rho$-Lie algebra,
		\item[iii.]
		$(\mathcal{B}, [.,.,.]_\mathcal{B},\rho,\phi)$ and $(\mathcal{B}^*, [.,.,.]_{\mathcal{B}^*},\rho,\phi^*)$ are 3-Hom-$\rho$-Lie subalgebras of $(\mathcal{B}\oplus \mathcal{B}^*, [.,.,.]_{\mathcal{B}\oplus \mathcal{B}^*},\rho,\phi+\phi^*)$.
	\end{enumerate}
\end{definition}
\begin{theorem}
A 3-Hom-$\rho$-Lie algebra has a phase space if and only if it is sub-adjacent to a 3-pre-Hom-$\rho$-Lie algebra.
\end{theorem}
\begin{proof}
By Theorem \ref{th.2020}, it is easy to get the result.
\end{proof}
\begin{lemma}
	Let $(V,\mu,\beta)$ be a representation of $\mathcal{B}_{\rho,\phi}$ and $(\mathcal{B}\oplus \mathcal{B}^*, [.,.,.]^{\mu^*}_{\mathcal{B}\oplus \mathcal{B}^*},\rho,\phi+\phi^{*}, \omega)$ be a phase space of it. Then the relation $\{f,g,h\}= \mu(f,g)h$ constructs a 3-pre-Hom-$\rho$-Lie algebra structure on $\mathcal{B}$ if $\beta=\phi$.
\end{lemma}
\begin{proof}
Assume that $f,g,h\in Hg(\mathcal{B})$ and $\alpha\in Hg(\mathcal{B}^*)$. Then, we have
\begin{align*}
\rho(f+g+h,\alpha)\alpha\circ\phi(\{f,g,h\})&=-\omega(\{f,g,h\},\alpha\circ\phi)=\rho(f+g,h)\omega(\phi(h),[f,g,\alpha]_{\mathcal{B}\oplus \mathcal{B}^*})\\
&=\rho(f+g,h)\omega(\phi(h),\mu^*(f,g)\alpha)\\
&=-\rho(f+g,h)\rho(h,f+g+\alpha)\mu^*(f,g)\alpha(\phi(h))\\
&= \rho(f+g+h,\alpha)\alpha\circ\phi\mu(f,g)h.
\end{align*}
\end{proof}
{\bf Data Availability Statements.} The data that support the findings of this study are available from the corresponding author upon reasonable request.
%%%%%%%%%%%%%%%%%%%%%%%%%%%%%%%%%%%%%%%%%%%%
%\section{Preliminaries}
% % % % % % % % % % % % % % % % % % % % % % % % % % % % % % % % % % % % % % % % % % % % % % 
%*******************************************************************************************

%%%%%%%%%%%%%%%%%%%%%%%%%%%%%%%%%%%%
%%%%%%%%%%%%%%%%%%%%%%%%%%%%%%%%%%%%%%%
\end{document}